\documentclass{amsart}

\usepackage{amssymb,mathtools}
\usepackage{microtype}
\usepackage[sort,nocompress]{cite}
\usepackage[hidelinks]{hyperref}
\makeatletter
\renewcommand{\paragraph}{\@startsection{paragraph}{4}%
  {\normalparindent}{\z@}{-\fontdimen2\font}%
  {\normalfont\bfseries}}
\makeatother
\newcommand{\doi}[1]{\href{https://doi.org/#1}{doi:\nolinkurl{#1}}}
\newcommand{\R}{\mathbb R}
\newcommand{\C}{\mathbb C}
\newcommand{\K}{\mathcal K}

\newcommand{\Q}{\mathcal Q}

\newcommand{\im}{\operatorname{Im}}
\newcommand{\re}{\operatorname{Re}}
\newcommand{\tr}{\operatorname{tr}}
\newcommand{\supp}{\operatorname{supp}}
\newcommand{\divg}{\operatorname{div}}

\newtheorem{theorem}{Theorem}[section]
\newtheorem{proposition}[theorem]{Proposition}
\newtheorem{lemma}[theorem]{Lemma}

\newtheorem*{proposition*}{Proposition}
\newtheorem*{lemma*}{Lemma}
\theoremstyle{remark}

\numberwithin{equation}{section}

\title[Threshold scattering for the focusing Schr\"odinger equation]{Threshold Scattering for the Three-Dimensional Focusing Energy-Critical Nonlinear Schr\"odinger Equation}

\author{Pang-hung Chung}
\address{Department of Mathematics, Guangdong University of Education, Guangzhou 510640, People's Republic of China}
\email{penghongzhong@yahoo.com}

\author{Dan Han}
\address{Department of Mathematics, University of Louisville, Louisville, Kentucky 40292, USA}
\email{dan.han@louisville.edu}
\thanks{Corresponding author: Dan Han.}
\hypersetup{pdfauthor={Pang-hung Chung and Dan Han}}

\date{}
\begin{document}
\raggedbottom

\begin{abstract}
Threshold scattering below the ground-state energy and gradient is established for the three-dimensional focusing energy-critical nonlinear Schr\"odinger equation. A single solution-dependent Gaussian interaction scale, mass-free angular coercivity, and a frequency-gap source estimate control the projected dynamics. These estimates yield quartic spacetime integrability and a zero-mass contradiction, without radial symmetry or an a priori $L^2$ bound.
\end{abstract}

\maketitle

\section{Introduction}
The competition between dispersion and concentration is a central problem
in nonlinear Schr\"odinger dynamics. Linear Schr\"odinger evolution spreads
a wave, whereas a focusing interaction can sustain a stationary profile or
produce concentration. At the energy-critical power, rescaling can change
the spatial and temporal concentration scales without changing the energy
norm. Energy conservation therefore does not by itself distinguish
scattering from a persistent concentration scenario. The sharp question is
whether the ground state marks the boundary between these behaviors.

In three dimensions this question is particularly delicate for nonradial
data in the homogeneous energy space. Such data need not have finite mass,
and a concentrated profile may move without any prescribed control of its
center. A proof must simultaneously recover positivity for the focusing
interaction, accommodate this motion, and control the low-frequency part
without assuming the mass that the argument is intended to construct.
The present paper gives a Gaussian interaction proof in which these issues
are separated into quantitative spatial estimates, a density-selected scale,
and an explicit estimate for the projected nonlinear source.

\paragraph{The equation and its threshold.}
We consider
\begin{equation}\label{NLS}
 (i\partial_t+\Delta)u=-|u|^4u,\qquad
 u(0)=u_0\in\dot H^1(\R^3),\qquad F(z):=|z|^4z.
\end{equation}
Here the homogeneous space is represented by $L^6$ functions with weak
gradient in $L^2$, and its norm is $\|f\|_{\dot H^1}:=\|\nabla f\|_2$.
The critical scaling is
$u(t,x)\mapsto a^{1/2}u(a^2t,ax)$, $a>0$.
For $f\in\dot H^1$, set
\[
 A(f):=\|\nabla f\|_2^2,\qquad V(f):=\|f\|_6^6,\qquad
 E(f):=\tfrac12 A(f)-\tfrac16 V(f).
\]
The stationary ground state is
$W(x):=(1+|x|^2/3)^{-1/2}$. With $G:=3^{3/2}\pi^2/4$, the sharp
Sobolev inequality and its normalization are
\begin{equation}\label{sharp}
 \forall f\in\dot H^1(\R^3),\qquad
 V(f)\le G^{-2}A(f)^3,\qquad
 A(W)=V(W)=G,\qquad E(W)=G/3.
\end{equation}
The sharp constant and optimizer are given in \cite{Aubin} and
\cite{Talenti}.
Indeed the sixth power of Talenti's constant is $16/(27\pi^4)=G^{-2}$.
The elliptic identity
\cite{KM}
gives $0<A(W)=V(W)$; equality in the Sobolev estimate gives
$A(W)=G^{-2}A(W)^3$. Positivity then gives $A(W)=G$ and $E(W)=G/3$.

A strong solution on an interval $I\ni0$ means a function
$u\in C(I;\dot H^1(\R^3))$ such that
$u\in L^{10}_{t,x}(J\times\R^3)$ for every compact interval $J\Subset I$
and such that Duhamel's formula holds on $I$. The continuation, stability
and scattering criteria used below are those of \cite{CW,Tao} and
\cite{KM}. For initial data
$u_0\in\dot H^1(\R^3)$, let $u:I_{\max}\times\R^3\to\C$ denote the unique
maximal-lifespan strong solution of \eqref{NLS}.

With the solution class fixed, the threshold-scattering statement is as follows.
\begin{theorem}\label{mainthm}
Let $u_0\in\dot H^1(\R^3)$ satisfy
\[
 E(u_0)<E(W),\qquad
 \|\nabla u_0\|_2<\|\nabla W\|_2.
\]
Then
\[
 I_{\max}=\R,
 \qquad
 \|u\|_{L^{10}_{t,x}(\R\times\R^3)}<\infty.
\]
Moreover, there exist $u_+,u_-\in\dot H^1(\R^3)$ such that
\[
 \lim_{t\to+\infty}\|u(t)-e^{it\Delta}u_+\|_{\dot H^1}=0,
 \qquad
 \lim_{t\to-\infty}\|u(t)-e^{it\Delta}u_-\|_{\dot H^1}=0.
\]
\end{theorem}

Thus the theorem gives global existence, finite critical spacetime size,
and scattering in both time directions in the homogeneous energy class;
no $L^2$ assumption is imposed on $u_0$. We next place this statement in
the existing scattering theory.

\paragraph{Development of the scattering theory.}
For the three-dimensional defocusing energy-critical equation, Colliander,
Keel, Staffilani, Takaoka and Tao \cite{CKSTT} established scattering for
general energy-space data through frequency-localized interaction Morawetz
estimates and induction on energy. The difference-variable interaction
can follow concentration without choosing a fixed spatial center.
For the focusing equation, Kenig and Merle \cite{KM} proved the
below-ground-state result for radial data in dimensions three through five,
using concentration compactness and a rigidity argument. Killip and Visan
\cite{KV5} removed radial symmetry in dimensions at least five through
additional decay and negative regularity. Dodson \cite{Dodson4} treated
dimension four using long-time Strichartz estimates to overcome the
logarithmic obstruction. These results identify both the role of the
threshold and the sensitivity of the low-frequency argument to dimension.

Recent work on the three-dimensional nonradial problem includes
Chung--Han's conditional rigidity under a center-speed hypothesis
\cite{CHLow}, and their exclusion of the finite-mass bounded-scale channel
\cite{CH}. The latter leaves a residual quasi-soliton channel and does not
by itself prove scattering for all homogeneous-energy data.
Su and Zhao \cite{SZ} also study this threshold-scattering problem in a recent
preprint. Related interaction and moving-scale methods appear in
\cite{PV,DM,DMMZ,BDF}.

\paragraph{The homogeneous three-dimensional obstacles.}
The threshold inequalities first provide an unweighted coercive gap.
Energy conservation, \eqref{sharp} and continuity prevent $A(u(t))$ from
reaching $G$, since at such a time the energy would be at least $G/3$.
The quantitative energy trap \cite{KM} consequently gives
\[
 \sup_t A(u(t))\le B^2\le(1-\gamma)G,\qquad
 A(u(t))-V(u(t))\ge(2\gamma-\gamma^2)A(u(t))
\]
for fixed $\gamma>0$. This global inequality does not directly control
a focusing interaction with a spatially varying anisotropic Hessian.
The negative potential term must be compared with the appropriate
directional derivatives at each relative separation.

There is also an infrared obstruction. The energy controls
$\nabla u$ in $L^2$, but the sum of low-frequency $L^2$ contributions
can diverge. In particular, even the ground state has infinite mass in
dimension three. A Galilean normalization of momentum would introduce
a multiple of $u$ into the gradient and is therefore unavailable without
additional control. A fixed-center virial cutoff encounters the separate
difficulty that almost periodicity does not furnish a velocity or a
sublinear-drift bound for the center. We instead apply a translation-invariant
interaction to the finite-mass highpass $h=P_{>\nu}u$.
Its mass depends on $\nu$, so recovering positivity by constants involving
that mass would not close the argument.

Finally, projection is not compatible with the quintic nonlinearity.
The defect $P_{>\nu}F(u)-F(P_{>\nu}u)$ contributes both momentum and mass
sources. A bound that is merely finite, or that presupposes the desired
spacetime norm, cannot be absorbed. The source estimate must use the
same quartic quantity and the same terminal intervals as the positive
interaction term, with a coefficient tending to zero as $\nu\downarrow0$.

\paragraph{The spatial estimates and the Gaussian choice.}
The first part of our construction converts the scalar threshold gap into
a usable two-point tensor inequality. Convex phase reduction expresses the
complex-to-modulus difference as an explicit nonnegative square.
The angular estimate then treats every nonnegative radial weight for which
$rw(r)$ is bounded and nondecreasing. Its proof combines a one-direction
dilation, a two-dimensional inversion, a change of the longitudinal
variable, and a normalized incidence-measure identity on great circles.
The resulting constant is explicit and independent of the $L^2$ size.
A complementary convolution estimate is obtained through
a regularized Fisher-type contraction. The two lemmas have different
roles: one controls tangential geometry; the other retains the full
convolution quadratic form.

These estimates are combined for a single Gaussian family. Its radial and tangential Hessian components fit the
two lemmas, while the density form has a nonnegative Fourier multiplier. Its physical-space kernel changes sign, so the
proof keeps the complete quadratic form throughout the core--tail
decomposition. The spatial coercivity estimate gives an explicit surviving fraction
$77/81$ and a strictly positive core contribution on a homogeneous
neighborhood of the normalized compact orbit. All cutoff and error choices
precede the radius bound, and none depends on the projected total mass.

For a large fixed radius, positivity alone would weaken the quartic
control. The density-dependent selector resolves this by choosing the same
Gaussian overlap level at every time. Its root is comparable with
$a(t)=\|h(t)\|_4^4$, and the spatial coercivity applies at the normalized
radius. More importantly, the Gaussian profile determining the root is
also the derivative of the action with respect to radius. The weak
continuity equation then gives an implicit identity whose unfavorable
part is bounded by $a(t)\|r(t)\|_1^2$. This is a signed estimate; a
small bound on the absolute speed of the radius is not needed.

\paragraph{A spatial frequency-gap estimate and closure.}
The projected output is handled by holding the complete
density weights fixed and splitting only the solution slot.
Exact Fourier-support separation removes every low-weight far interaction.
The remaining three terms use the nonendpoint mappings
$I_1:L^2\to L^6$ and $I_2:L^{4/3}\to L^{12}$, followed by integer-order
highpass estimates. Balancing the near and far losses gives
\[
 \int_{I_j}|\mathfrak S_k(\mathfrak o_\nu,h)|\,dt
 \le C_u\sqrt{\delta_\nu}\,\mathfrak K_\nu(I_j),\qquad
 \delta_\nu:=\nu\sup_t\|P_{>\nu}u(t)\|_2\longrightarrow0.
\]
Local endpoint Strichartz remains necessary for the energy-class limit and
the final recovery of mass.

We first construct the actual $L^4$ and $L^1$ forcing
representatives before any time-integrability claim.
We then select a cutoff-independent sequence $T_j$
and control the forcing and its future maximal function by the projected
quartic integral. The input cancellation, output frequency gap and signed
radius estimate therefore fit one finite-interval absorption.
The complete identity is passed to the original weak energy
solution while retaining the original, unsmoothed selector.
Only after this absorption do we obtain $F(u)\in L^2_tL^1_x$.
For each fixed lowpass frequency, dual endpoint Strichartz and no-waste
Duhamel construct an actual $L^2$ value. Mass conservation is invoked after
that construction. The time limit is taken first and the frequency limit
second, forcing the mass to vanish and excluding the critical element.
The ordinary finite-size criterion then supplies both scattering states.

\paragraph{Conventions.}
Unlabelled spatial integrals are over $\R^3$, and
\[
 \widehat f(\xi):=\int e^{-ix\cdot\xi}f(x)\,dx,\qquad
 f(x)=(2\pi)^{-3}\int e^{ix\cdot\xi}\widehat f(\xi)\,d\xi,\qquad
 A:B:=\tr(A^TB).
\]
For vectors $a,b\in\R^3$ or $\C^3$, $a\otimes b$ denotes the outer
product, with $(a\otimes b)_{ij}:=a_i b_j$; complex conjugation is written
explicitly. The symbol $I$ denotes the identity operator or matrix on the
space in question. For real symmetric matrices, $A\succeq B$ means
$\xi^T(A-B)\xi\ge0$ for every real vector $\xi$.
We write $\partial_\omega f:=\omega\cdot\nabla f$ for the directional
weak derivative. The spaces $\mathcal S$ and $\mathcal S'$ are the Schwartz
space and its tempered-distribution dual, and $\mathbf1_E$ is the indicator
of a set $E$. A norm $\|\cdot\|_p$ without time variables is a spatial
$L^p$ norm. Fix a real radial $\varphi\in C_c^\infty(\R^3)$ with
$0\le\varphi\le1$, $\varphi|_{B_1}=1$ and
$\supp\varphi\subset\overline B_2$. For every $L>0$, define
\[
 \widehat{P_{\le L}f}(\xi):=\varphi(\xi/L)\widehat f(\xi),\qquad
 P_{>L}:=I-P_{\le L},\qquad P_L:=P_{\le L}-P_{\le L/2}.
\]
Frequency sums use dyadic scales. Each estimate retains its displayed
frequency domain. Implicit constants are independent of the projection
frequency, terminal time and selected Gaussian scale unless a dependence
is stated.

\paragraph{Organization.}
Section~2 gives the compactness reduction and interaction structure,
Section~3 the forcing comparison, Section~4 coercivity and scale selection,
and Section~5 the source estimates, energy-class limiting identity,
and rigidity argument.

\section{Compactness Reduction and a Moving Interaction Functional}
\subsection{A normalized nonscattering orbit}
Suppose Theorem~\ref{mainthm} fails. The profile decomposition
\cite{Keraani} and \cite{KM} give
a maximal critical solution with a compact normalized orbit on one
half-lifespan of infinite scattering size. If that direction is negative,
replace $u(t,x)$ by $\overline{u(-t,x)}$. Denote the resulting maximal
lifespan by $I\ni0$ and set $J:=I\cap[0,\infty)$,
$T_+:=\sup I\in(0,\infty]$.
Let $N:I\cap(0,\infty)\to(0,\infty)$ and
$x:I\cap(0,\infty)\to\R^3$ be the frequency and center modulations.
Extend the modulation to $t=0$ by $N(0)=1$, $x(0)=0$, adjoining $u(0)$
to the compact profile set. The resulting data satisfy
\begin{equation}\label{critical-data}
 \begin{gathered}
 0<E_c<E(W),\qquad E(u(t))=E_c,\qquad
 \sup_{t\in I}A(u(t))\le B^2<G,\\
 \mathcal K_0:=
 \overline{\{N(t)^{-1/2}u(t,x(t)+\cdot/N(t)):t\in J\}}^{\dot H^1}
 \text{ compact},\qquad 0\notin\mathcal K_0,
 \end{gathered}
\end{equation}
with $\|u\|_{L^{10}(J\times\R^3)}=\infty$. Spatial
translations occur in those propositions; the separate radial restriction
is \cite{KM}.

\emph{Symmetries and energy.} For $a>0$, $t_0\in I$ and $x_0\in\R^3$, put
\begin{align*}
 (\mathcal T_{a,t_0,x_0}u)(s,y)
 &:=a^{-1/2}u(t_0+s/a^2,x_0+y/a),
 & I_{\mathcal T u}&:=\{s:t_0+s/a^2\in I\},\\
 (\mathcal Ru)(s,y)&:=\overline{u(-s,y)},
 &I_{\mathcal Ru}&:=-I.
\end{align*}
The changes of variables in the weak derivatives and mild identity show
that both transforms solve \eqref{NLS}. They preserve the energy-space
quantities used below:
\[
 A(\mathcal Tu)=A(u),\qquad
 E(\mathcal Tu)=E(u)=E(\mathcal Ru).
\]
These are \cite{Tao}, extended by local approximation and uniqueness.
Moreover, if $f_n\to f$ in $\dot H^1$, then \eqref{sharp} gives
$f_n\to f$ in $L^6$, and therefore
\[
 E(f_n)\longrightarrow E(f).
\]
All limits below therefore retain the same $E_c$.

\emph{Vanishing linear endpoint.} Fix $t\in I$ and
$\psi\in C_c^\infty(\R^3)$. Compactness in \eqref{critical-data} gives
\[
 \forall\varepsilon>0\ \exists C_\varepsilon>0\ \forall T\in J:\quad
 \|\mathbf1_{|x-x(T)|>C_\varepsilon/N(T)}\nabla u(T)\|_2\le\varepsilon.
\]
With $r_T:=C_\varepsilon/N(T)$, Cauchy--Schwarz and unitarity give
\begin{align}
 \bigl|\langle\nabla u(T),e^{i(T-t)\Delta}\nabla\psi\rangle\bigr|
 &\le\varepsilon\|\nabla\psi\|_2
 +CB r_T^{3/2}\|e^{i(T-t)\Delta}\nabla\psi\|_\infty.
 \label{endpoint-pairing}
\end{align}
At a finite endpoint, local constancy gives $N(T)\to\infty$ and the
test evolution is bounded in $L^\infty$. At an infinite endpoint with
$N(T)\ge1$, dispersion gives $\|e^{i(T-t)\Delta}\nabla\psi\|_\infty
\le C_\psi|T-t|^{-3/2}$. Thus, letting $T\to T_+$ in
\eqref{endpoint-pairing} gives, in either case,
\[
 \limsup_{T\to T_+}|\langle\nabla u(T),e^{i(T-t)\Delta}\nabla\psi\rangle|
 \le\varepsilon\|\nabla\psi\|_2.
\]
Letting $\varepsilon\downarrow0$, and using the density of $C_c^\infty$
in $\dot H^1$ together with $\sup_T\|\nabla u(T)\|_2\le B$, we obtain
\[
 e^{i(t-T)\Delta}u(T)\rightharpoonup0\quad\text{in }\dot H^1
 \quad(T\to T_+).
\]
The mild identity consequently yields the sign-independent no-waste
formula of \cite{KV3}:
\begin{equation}\label{no-waste-formula}
 u(t)=-i\operatorname*{w-lim}_{T\to T_+}
 \int_t^T e^{i(t-s)\Delta}F(u(s))\,ds
 \quad\text{in }\dot H^1.
\end{equation}

\emph{Finite endpoints.} If $T_+<\infty$, then, for each $t\in I$,
\begin{align*}
 \|F(u)\|_{L^2([t,T_+);L^{6/5})}
 &\le CB^5(T_+-t)^{1/2},\\
 w_t:=-i\int_t^{T_+}e^{i(t-s)\Delta}F(u(s))\,ds
 &\in L^2,\qquad \|w_t\|_2\le CB^5(T_+-t)^{1/2}.
\end{align*}
The second line is dual endpoint Strichartz
\cite{KT,KM}.
The no-waste formula \eqref{no-waste-formula} identifies the weak derivatives:
\[
 \partial_jw_t=\partial_ju(t)\quad(j=1,2,3)\quad\text{in }\mathcal S'.
\]
Since $w_t\in L^2$ and $\nabla w_t=\nabla u(t)\in L^2$, one has
$w_t\in H^1\subset L^6$. The difference $w_t-u(t)$ belongs to $L^6$ and
has zero weak gradient, so $w_t=u(t)$. The standard $H^1$ mass law then gives
\[
 \|u(t_0)\|_2=\|u(t)\|_2
 \le CB^5(T_+-t)^{1/2}\longrightarrow0.
\]
Consequently $u(t_0)=0$ and $E_c=0$, contradicting \eqref{critical-data}.
The endpoint in the selected direction therefore satisfies $T_+=\infty$.
This agrees with \cite{KV5}; finite mass is a conclusion
of the displayed $L^2$ construction.

\emph{A lower frequency bound.} Use the normalized-limit selection of
\cite{KV5}, or the half-lifespan normalization preceding
\cite{KV3}. If $\inf_{t\ge0}N(t)>0$, one fixed dilation suffices.
Otherwise choose
\[
 t_n\to\infty,\qquad N_n:=N(t_n)\to0,\qquad
 N_n\le2\inf_{0\le t\le t_n}N(t),\qquad d_n:=t_nN_n^2,
\]
and define
\begin{align*}
 v_n(s,x)&:=N_n^{-1/2}u(t_n+s/N_n^2,x(t_n)+x/N_n),
       &&s\ge-d_n,\\
 N_{v_n}(s)&=N(t_n+s/N_n^2)/N_n\ge\tfrac12,
       &&-d_n\le s\le0,\\
 v_n(0)&\longrightarrow v(0)\quad\text{in }\dot H^1,
       &&E(v)=E_c,\quad\sup A(v)\le B^2.
\end{align*}
The last line, after subsequence extraction, uses compactness and stability
\cite{KM}. On each inherited compact
time interval the convergence is in $C_t\dot H^1$, with the same compactness
modulus and frequency lower bound. The native mild solution is identified
by homogeneous uniqueness in the critical local theory.

To prove $d_n\to\infty$, suppose along a subsequence that $d_n\to d_*<\infty$.
Finite-endpoint exclusion and stability give
\begin{align*}
 v_n(-d_n)&\longrightarrow v(-d_*)\quad\text{strongly in }\dot H^1,\\
 v_n(-d_n,x)&=N_n^{-1/2}u(0,x(t_n)+x/N_n)
              \rightharpoonup0\quad\text{in }\dot H^1.
\end{align*}
The second line is weak escape of critical dilates of the fixed datum
$u(0)$, uniformly in translations. Uniqueness of the weak limit yields
$v(-d_*)=0$, hence $E_c=0$, contradicting \eqref{critical-data}.
Hence $v$ exists on $(-\infty,0]$ with $N_v\ge1/2$; time reversal and one
fixed dilation give a forward-global solution satisfying
\begin{equation}\label{orbit}
 \K=\overline{\{N(t)^{-1/2}u(t,x(t)+\cdot/N(t)):t\ge0\}}^{\dot H^1}.
\end{equation}
The standing orbit data, used from now on, are
\begin{equation}\label{standing-orbit-data}
 \begin{gathered}
 t\ge0,\quad N(t)\ge1,\quad x(t)\in\R^3,\quad
 \K\text{ compact},\quad 0\notin\K,\\
 E(u(t))=E_c\in(0,E(W)),\qquad
 \sup_{v\in\K}A(v)\le B^2<G,\qquad
 \|u\|_{L^{10}([0,\infty)\times\R^3)}=\infty.
 \end{gathered}
\end{equation}
In particular \eqref{no-waste-formula} now holds with $T_+=\infty$.
For fixed $0<\nu<1$, define
\[
 h:=P_{>\nu}u,\qquad l:=u-h=P_{\le\nu}u,\qquad
 a(t):=\|h(t)\|_4^4,\qquad h\in C([0,\infty);H^1).
\]
The moving scale below is selected from $h$.

\subsection{Projected interaction balance}
For an even real kernel \(k\), set
\[
 \rho:=|h|^2,\qquad
 j:=\im(\bar h\nabla h),\qquad
 \mathbf T:=\re(\nabla\bar h\otimes\nabla h),
\]
and define
\begin{align*}
 Q_k(h)&:=-\iint\Delta^2k(x-y)\rho(x)\rho(y)\,dxdy,\\
 M_k(h)&:=2\iint \rho(y)\nabla k(x-y)\cdot j(x)\,dxdy.
\end{align*}
Introduce the two-point stress tensor
\[
 \mathbf S_h(x,y):=
 \rho(y)\mathbf T(x)-j(y)\otimes j(x)
 -\frac13|h(x)|^6\rho(y)I .
\]
The interaction principal form is then compressed to
\[
 \mathfrak B_k[h]
 :=Q_k(h)+4\iint D^2k(x-y):\mathbf S_h(x,y)\,dxdy .
\]
For the fixed projection frequency \(\nu\), let
\[
 E:=P_{>\nu}F(u)-F(h),\qquad
 q:=\nabla k*\rho,\qquad
 q_m(y):=\int\nabla k(x-y)\cdot j(x)\,dx.
\]
The forcing contribution is denoted by
\[
 \mathfrak S_k(E,h)
 :=-2\int q\cdot\re(E\nabla\bar h-h\nabla\bar E)
   -4\int q_m\,\im(E\bar h).
\]
For smooth fields the projected local conservation laws give
\[
 \frac{d}{dt}M_{k_{\lambda(t)}}(h)
 =\mathfrak B_{k_{\lambda(t)}}[h]
  +\mathfrak S_{k_{\lambda(t)}}(E,h)
  +\lambda'(t)\mathcal J_\lambda,
 \qquad
 \mathcal J_\lambda
 :=2\int j\cdot[(\partial_\lambda\nabla k_\lambda)*\rho].
\]
See \cite{PV,DM,BDF}. The same identity for energy-class solutions, with the same
forcing and moving-scale term, will be recovered below by smooth approximation.
The first coercive step is the following phase reduction.

\begin{lemma}[Convex phase reduction]\label{phase}
Let $h\in H^1(\R^3;\C)$ and let $k$ be an even real $C^4$ kernel.
Assume that $Q_k(h)=Q_k(|h|)$ is finite and that both scalar stress
contractions defining $\mathfrak B_k[h]$ and $\mathfrak B_k[|h|]$
are absolutely integrable on $\R^3\times\R^3$.
If \(D^2k\succeq0\), then
\[
 \mathfrak B_k[h]\ge \mathfrak B_k[|h|].
\]
\end{lemma}
\begin{proof}
Define, for almost every $x\in\R^3$,
\[
 g:=|h|,\qquad
 \zeta(x):=\begin{cases}h(x)/g(x),&g(x)>0,\\0,&g(x)=0,\end{cases}
 \qquad \Lambda:=\im(\bar\zeta\nabla h).
\]
Polar factorization \cite{ZhengPolar}, with
$\hbar=1$, and its zero-set assertion \cite{ZhengPolar} give
\begin{equation}\label{polar-identities}
 \begin{gathered}
 g\in H^1,\quad\Lambda\in L^2,\quad
 j=g\Lambda,\quad
 \mathbf T=\nabla g\otimes\nabla g+\Lambda\otimes\Lambda,\\
 \Lambda=0\quad\text{a.e. on }\{g=0\}.
 \end{gathered}
\end{equation}
On $\R^3\times\R^3$, set
\begin{align*}
 H(x,y)&:=D^2k(x-y)=H(y,x),\\
 Z(x,y)&:=g(y)\Lambda(x)-g(x)\Lambda(y),\\
 D(x,y)&:=H(x,y):
 \bigl[g(y)^2\Lambda(x)\otimes\Lambda(x)
       -g(x)g(y)\Lambda(y)\otimes\Lambda(x)\bigr].
\end{align*}
The two integrable stress contractions in the hypotheses imply
$D\in L^1(\R^3\times\R^3)$. Consequently,
\begin{align*}
 \mathfrak B_k[h]-\mathfrak B_k[g]
 &=4\iint D(x,y)\,dxdy,
\\
 &=2\iint\bigl[D(x,y)+D(y,x)\bigr]\,dxdy,\\
 &=2\iint H(x,y):[Z(x,y)\otimes Z(x,y)]\,dxdy,\\
 &\ge0.
\end{align*}
For each Gaussian kernel used below, the required integrability follows from
\begin{gather*}
 \|D^2k_R\|_\infty+\|\Delta^2k_R\|_\infty<\infty,\qquad
 \|j\|_1\le\|h\|_2\|\nabla h\|_2,\qquad h\in H^1\subset L^6,\\
 |Q_{k_R}(h)|\le\|\Delta^2k_R\|_\infty\|h\|_2^4,\\
 \iint|D^2k_R:\mathbf S_h|
 \lesssim\|D^2k_R\|_\infty
 \bigl(\|h\|_2^2\|\nabla h\|_2^2+\|j\|_1^2
       +\|h\|_6^6\|h\|_2^2\bigr)<\infty.
\end{gather*}
The same bounds apply to $g$ by \eqref{polar-identities}.
\end{proof}

This removes the complex phase from the principal stress without weakening the interaction. We next record the translated Hardy estimate needed to control the resulting singular pairings.
For the fixed projected field $h=P_{>\nu}u\in H^1$, translated Hardy
\cite{DietzeNam} and Cauchy--Schwarz give
\begin{equation}\label{Hardy}
 \begin{aligned}
 \sup_x\int\frac{|h(y)|^2}{|x-y|}\,dy
 &\le\|h\|_2\sup_x\bigl\||x-\cdot|^{-1}h\bigr\|_2\\
 &\le2\|h\|_2\|\nabla h\|_2.
 \end{aligned}
\end{equation}
No uniform projected-mass bound is used in the coercivity constants.

\section{Low-Frequency Control and Time Selection}
\subsection{A mass-free forcing envelope}
The refined Sobolev inequality
\cite{MGO}
applies to $v\in\dot H^1$ with finite displayed Besov seminorm and has the form
\begin{equation}\label{refined}
 \|v\|_4^4\le C\left(\sup_{L>0}L^{-1}\|P_Lv\|_\infty\right)^2
                  \|\nabla v\|_2^2.
\end{equation}
The bound over all $L>0$ includes the dyadic Besov bound in this refinement.
This refinement yields the mass-free forcing envelope needed for the time selection.
\begin{lemma}\label{L4lemma}
For the selected critical element,
\[
 u\in L^\infty_tL^4_x,\qquad
 F(u)\in L^\infty_tL^1_x .
\]
\end{lemma}
\begin{proof}
\emph{Finite frequency rows.} For $L>0$, define
\[
 b(L):=L^{-1}\sup_{t\ge0}\|P_Lu(t)\|_\infty,\qquad
 A_*:=\sup_{L>0}b(L)\in[0,\infty].
\]
Bernstein and compactness give, in this order,
\begin{align}
 b(L)&\le CB L^{-1/2}<\infty,\qquad L>0,\label{band-initial}\\
 v&:=P_{\le L_0}u,\qquad w:=u-v,\\
 \sup_{t\ge0}\|\nabla v(t)\|_2&\le\eta,\label{small-v}
\end{align}
where, for each $\eta>0$, the frequency $L_0>0$ is chosen by compactness.
For the corresponding $L_0$ and every $t\ge0$, the high-pass estimate,
Sobolev embedding, interpolation, and the quintic H\"older bound give
\begin{align*}
 \|w\|_2&\lesssim L_0^{-1}B,\qquad \|w\|_6\lesssim B,\\
 \|w\|_3&\le\|w\|_2^{1/2}\|w\|_6^{1/2}
       \lesssim L_0^{-1/2}B,\\
 \|F(u)-F(v)\|_1
 &\lesssim\|w\|_3(\|u\|_6+\|v\|_6)^4
       \le C_{L_0,B}.
\end{align*}
For fixed $L>0$, put
\begin{gather*}
 s:=L/100,\qquad r:=P_{\le s}v,\qquad z:=P_{>s}v,\qquad v=z+r,\\
 H_L:=\sup_{M\ge L/200}b(M),\qquad
 D_L:=\frac1{100}\sum_{n\ge0}2^{-n}b(s2^{-n}).
\end{gather*}
Equation~\eqref{band-initial} gives $H_L,D_L<\infty$. If $M\le s/2$,
then $P_Mz=0$. If $M>s/2$, commutation gives
\[
 P_Mz=P_{>s}P_{\le L_0}P_Mu,
 \qquad
 M^{-1}\|P_Mz\|_\infty\lesssim H_L.
\]
Telescoping in $L^\infty$ gives
\begin{align*}
 r&=\sum_{n\ge0}P_{s2^{-n}}v,\\
 \sum_{n\ge0}\|P_{s2^{-n}}v\|_\infty
 &\le CB\sum_{n\ge0}(s2^{-n})^{1/2}<\infty,\\
 \sup_{M>0}M^{-1}\|P_Mz\|_\infty&\lesssim H_L,
 \qquad L^{-1}\|r\|_\infty\lesssim D_L.
\end{align*}
The telescoping remainder vanishes by \eqref{band-initial}.

\emph{An actual $L^1$ source.} The multiplier bounds and \eqref{small-v} give
$\|\nabla z\|_2+\|\nabla r\|_2\lesssim\eta$ and
$\|z\|_2\lesssim s^{-1}\eta$. Applying \eqref{refined} to $z$ gives
\begin{align*}
 \|F(z)\|_1
 &\le\|z\|_4^2\|z\|_6^3
 \lesssim (H_L\eta)\eta^3
 =\eta^4H_L,\\
 |F(z+r)-F(z)-F(r)|
 &\le C|r||z|(|z|+|r|)^3,\\
 \|F(v)-F(z)-F(r)\|_1
 &\le C\|r\|_\infty\|z\|_2(\|z\|_6+\|r\|_6)^3
 \lesssim \|r\|_\infty s^{-1}\eta^4
 \lesssim\eta^4D_L.
\end{align*}
Define the actual function
\[
 G_L:=F(u)-F(v)+F(z)+[F(v)-F(z)-F(r)]=F(u)-F(r).
\]
The preceding inequalities imply
\begin{equation}\label{L4source}
 \sup_{t\ge0}\|G_L(t)\|_1\le C_{L_0,B}+C\eta^4(H_L+D_L),
 \qquad P_LG_L=P_LF(u).
\end{equation}
For the second identity, the product-support rule gives
\[
 \supp\widehat r\subset\{|\xi|\le2s\},
 \qquad
 \supp\widehat{F(r)}\subset\{|\xi|\le10s\}.
\]
Since $10s=L/10<L/2$,
\[
 P_LF(r)=0,
 \qquad
 P_LG_L=P_LF(u).
\]
Here the product/convolution support rules are
\cite{Dyatlov}.

\emph{No-waste recurrence.} For $\tau>0$,
\begin{gather*}
 \|e^{-i\tau\Delta}P_L\|_{L^1\to L^\infty}
 \lesssim\min\{L^3,\tau^{-3/2}\},\\
 \int_0^\infty\min\{L^3,\tau^{-3/2}\}\,d\tau
 =L^3L^{-2}+\int_{L^{-2}}^\infty\tau^{-3/2}\,d\tau=3L.
\end{gather*}
For fixed $L$, Schwartz-kernel continuity and Fourier inversion
\cite{Dyatlov},
followed by Young \cite{Tao}, imply
\begin{gather*}
 \tau\longmapsto e^{-i\tau\Delta}P_L
       \in C\bigl(\R;\mathcal L(L^{6/5},L^\infty)\bigr),\qquad
 F(u)\in C_tL^{6/5}_x,\\
 \tau\longmapsto e^{-i\tau\Delta}P_LF(u(t+\tau))
       \quad\text{strongly measurable},\\
 \int_0^\infty\|e^{-i\tau\Delta}P_LG_L(t+\tau)\|_\infty\,d\tau
 \lesssim L\bigl[C_{L_0,B}+C\eta^4(H_L+D_L)\bigr]<\infty.
\end{gather*}
Thus the Bochner integral is defined. Its distributional identity with
\eqref{no-waste-formula}, using \eqref{L4source}, gives
\begin{equation}\label{L4recurrence}
 b(L)\le K+C\eta^4(H_L+D_L),\qquad K:=C_{L_0,B},\qquad L>0.
\end{equation}

\emph{Finite-supremum absorption.} Define
\begin{gather*}
 \alpha:=\tfrac34,\qquad
 w_\varepsilon(L):=\min\{1,(L/\varepsilon)^\alpha\},\qquad
 A_\varepsilon:=\sup_{L>0}w_\varepsilon(L)b(L),\qquad\varepsilon>0,\\
 A_\varepsilon\lesssim
 B\varepsilon^{-1/2}<\infty,\qquad
 C_\alpha:=200^\alpha+
 \frac{100^{\alpha-1}}{1-2^{-(1-\alpha)}}.
\end{gather*}
Combining \eqref{L4recurrence} with comparison of the weights and the geometric series gives
\[
 w_\varepsilon(L)(H_L+D_L)\le C_\alpha A_\varepsilon.
\]
Choose $\eta$ so that $C\eta^4C_\alpha\le\tfrac12$. Then
\[
 A_\varepsilon\le K+\tfrac12A_\varepsilon,
 \qquad
 A_\varepsilon\le2K.
\]
Taking $\varepsilon=L$ and then the supremum over $L$ gives
\[
 b(L)\le2K,
 \qquad
 A_*\le2K<\infty.
\]
Consequently,
\begin{align*}
 \sup_{t\ge0}\|u(t)\|_4^4&\le CA_*^2B^2<\infty,\\
 \sup_{t\ge0}\|F(u(t))\|_1&<\infty.
\end{align*}
The parameter order is: choose $\eta$ satisfying the second line,
then choose $L_0$ by \eqref{small-v}, then fix $K$; only afterwards vary
$\varepsilon$ and $L$.

\emph{Bochner representatives.} For every $t_0\ge0$,
\begin{align*}
 \sup_{L>0}L^{-1}\|P_L(u(t)-u(t_0))\|_\infty&\le2A_*,\\
 \|u(t)-u(t_0)\|_4^4
 &\lesssim
 A_*^2\|\nabla u(t)-\nabla u(t_0)\|_2^2\longrightarrow0.
\end{align*}
Continuity and the displayed uniform bounds give the two genuine
Bochner $L^\infty_t$ memberships in the statement.
\end{proof}
The lemma supplies both the bounded forcing amplitude and the time continuity required for the terminal-time argument. Set
\[
 f(t):=\|F(u(t))\|_1,\qquad U_4:=\|u\|_{L^\infty_tL^4_x},\qquad
 F_*:=\|F(u)\|_{L^\infty_tL^1_x}.
\]

\subsection{Small projections and admissible terminal times}
For $\nu>0$, define
\[
 \eta_\nu:=\sup_{t\ge0}\|\nabla P_{\le\nu}u(t)\|_2,
 \qquad\mathfrak m_\nu:=\sup_{t\ge0}\|P_{>\nu}u(t)\|_2,
 \qquad\delta_\nu:=\nu\mathfrak m_\nu.
\]
Compactness of \eqref{orbit} and $N\ge1$ give $\eta_\nu\to0$.
For every $c,\nu>0$, the exact smooth-multiplier identity yields
\begin{align*}
 P_{>\nu}u
 &=P_{>\nu}P_{\le c}u+(I-P_{\le\nu})P_{>c}u,\\
 \delta_\nu&\le C\eta_c+2C(\nu/c)B,\\
 \limsup_{\nu\downarrow0}\delta_\nu&\le C\eta_c.
\end{align*}
Since $c>0$ is arbitrary and $\eta_c\to0$, one obtains $\delta_\nu\to0$.
The two high-pass terms are actual $L^2$ functions; their $L^6$
representatives agree in $\mathcal S'$. No smooth-cutoff idempotence is used.
Thus
\begin{equation}\label{smallfreq}
 \varepsilon_{\rm lf}(\nu):=\eta_\nu+\delta_\nu\longrightarrow0.
\end{equation}
For the normalized profile
$\widetilde h_t(x):=N(t)^{-1/2}h(t,x(t)+x/N(t))$, the projection
error and continuity of the energy give, uniformly in $t\ge0$,
\[
 \operatorname{dist}_{\dot H^1}(\widetilde h_t,\K)\le\eta_\nu,
 \qquad |E(h(t))-E(u(t))|\le C(B+B^5)\eta_\nu.
\]
By \eqref{smallfreq}, these errors tend to zero as $\nu\downarrow0$.
The strict gradient and energy gaps in \eqref{standing-orbit-data}
therefore persist for $h$: the gradient norm changes by at most
$\eta_\nu$, while the energy changes by the displayed bound.
Choosing $\nu$ sufficiently small, independently of $t$, preserves both
threshold inequalities. Critical rescaling preserves the same gaps for
$\widetilde h_t$.
Compactness, $0\notin\K$, local Sobolev and rescaling give
\begin{gather*}
 \exists R_0,c_0>0\ \forall v\in\K:\quad\|v\|_{L^4(B_{R_0})}\ge c_0,
 \qquad\|\widetilde h_t\|_{L^4(B_{R_0})}\ge c_0/2\quad(\nu\text{ small}),\\
 \inf_{t\ge0}N(t)a(t)>0,\qquad
 \sup_{t\ge0}a(t)<\infty.
\end{gather*}
Only the lower bound on $N(t)a(t)$ is used.

For $t\ge0$ and a bounded interval $I\subset[0,\infty)$, define
\[
 \mathcal I_t:=\{[t,t+2^m]:m\in\mathbb Z\},\qquad
 \mathcal Y(t):=\sup_{I\in\mathcal I_t}|I|^{-1}\int_I f,
 \qquad\mathfrak K_\nu(I):=\int_Ia(t)\,dt.
\]
The column $\begin{pmatrix}f\\\mathcal Y\end{pmatrix}$ denotes the
$\R^2$-valued function $t\mapsto(f(t),\mathcal Y(t))$, with the Euclidean
norm. In particular, on every bounded interval $J$,
\[
 \left\|\begin{pmatrix}f\\\mathcal Y\end{pmatrix}\right\|_{L^2(J;\R^2)}^2
 =\int_J\bigl(f(t)^2+\mathcal Y(t)^2\bigr)\,dt.
\]
We now select terminal times on which the forcing envelope and its future maximal function are controlled by the projected quartic quantity.
\begin{proposition}\label{comparison}
There is a sequence $T_j\uparrow\infty$, independent of $\nu$, such that for all sufficiently small $\nu$,
\begin{equation}\label{fcompare}
 \left\|\begin{pmatrix}f\\\mathcal Y\end{pmatrix}\right\|_{L^2(0,T_j;\mathbb R^2)}
 \le C_u\mathfrak K_\nu([0,T_j])^{1/2}.
\end{equation}
Moreover,
\begin{equation}\label{lowpoint}
 \sup_{L>0}\left[
 L^{-1}\|P_Lu(t)\|_\infty+L^{-2}\|\nabla P_Lu(t)\|_\infty
 \right]\le C\mathcal Y(t).
\end{equation}
Dyadic summation gives the corresponding low-pass bounds at every positive frequency.
\end{proposition}
\begin{proof}
\emph{Maximal functions.} Extend $f$ by zero on $(-\infty,0)$ and put
\[
 \mathcal M^+f(t):=\sup_{s>0}s^{-1}\int_t^{t+s}f(\tau)\,d\tau.
\]
Every positive window is contained in a dyadic window of length at most
twice its length. Therefore
\begin{equation}\label{dyadicfull}
 0\le\mathcal Y(t)\le\mathcal M^+f(t)\le2\mathcal Y(t)\le2F_*.
\end{equation}
Continuity of the primitive gives the countable representation
\[
 \mathcal M^+f(t)=\sup_{s\in\mathbb Q_{>0}}s^{-1}\int_t^{t+s}f;
 \qquad \mathcal Y,\mathcal M^+f\text{ are measurable}.
\]

\emph{Terminal times, before $\nu$.} For $T>0$, set
\[
 E_f(T):=\int_0^Tf^2,\qquad A_f(T):=E_f(T)/T\in[0,F_*^2].
\]
For each $R>0$, let $S_R:=\sup_{s\ge R}A_f(s)$. Choose $T\ge R$ by
\[
 S_R=0:\ T:=R;\qquad
 S_R>0:\ A_f(T)>S_R/2.
\]
Then
\[
 A_f(s)\le2A_f(T)\qquad(s\ge T).
\]
Choose $R\ge\max\{j+1,T_{j-1}+1\}$. The resulting terminal times satisfy
\[
 T_j>T_{j-1},\qquad T_j\to\infty,
 \qquad E_f(2T_j)\le4E_f(T_j).
\]
This construction depends on $f$ only.

\emph{An $L^2$ bound on each selected interval.} Fix $T=T_j$.
For $0\le t\le T$ and $0<s\le T$, one has $[t,t+s]\subset[0,2T]$.
If $s>T$, then
\[
 \left(s^{-1}\int_t^{t+s}f\right)^2
 \le E_f(t+s)/s
 \le 2\frac{t+s}{s}\frac{E_f(T)}T
 \le4\frac{E_f(T)}T.
\]
The first line uses $f\mathbf1_{[0,2T]}$. The one-sided maximal function
is bounded by twice the centered maximal function; hence
\cite{TaoMax} gives
\begin{equation}\label{selected-maximal-L2}
 \begin{aligned}
 \int_0^T(\mathcal M^+f)^2
 &\le C_{\rm HL}E_f(2T)+4E_f(T)\\
 &\le(4C_{\rm HL}+4)E_f(T).
 \end{aligned}
\end{equation}

\emph{All positive spatial frequencies.} For $L>0$, set
$k_L(s):=L^{-1}\min\{L^3,s^{-3/2}\}$, $s>0$. Then
\begin{gather*}
 k_L\ge0,\quad k_L\text{ decreasing},\quad\int_0^\infty k_L=3,
 \qquad\int_0^rf(t+s)\,ds\le r\mathcal M^+f(t),\\
 \int_0^\infty k_L(s)f(t+s)\,ds
 \le\mathcal M^+f(t)\int_0^\infty k_L(s)\,ds
 =3\mathcal M^+f(t).
\end{gather*}
The second line follows by positive layer integration. Applying it to the
actual absolutely convergent annular integral in Lemma~\ref{L4lemma},
\begin{align*}
 L^{-1}\|P_Lu(t)\|_\infty
 &\lesssim\int_0^\infty k_L(s)f(t+s)\,ds
 \lesssim\mathcal Y(t),\\
 P_{\le L}u&=\sum_{n\ge0}P_{L2^{-n}}u\quad\text{in }L^\infty
 \qquad\|P_{\le L}u(t)\|_\infty\lesssim L\mathcal Y(t),\\
 \|\nabla P_{\le L}u(t)\|_\infty
 &\lesssim L\|P_{\le L}u(t)\|_\infty
 \lesssim L^2\mathcal Y(t).
\end{align*}
Since $P_L=P_{\le L}-P_{\le L/2}$, these estimates prove
\eqref{lowpoint}, including its bound for the actual weak gradient.

\emph{Uniform absorption.} Let $l=P_{\le\nu}u$, $h=u-l$ and
$C_M:=(4C_{\rm HL}+4)^{1/2}$. Then
\begin{align*}
 \sup_{L>0}L^{-1}\|P_Ll(t)\|_\infty
 &\lesssim\mathcal M^+f(t),\\
 \|l(t)\|_4^2
 &\lesssim
 \eta_\nu\mathcal M^+f(t),\qquad\|l(t)\|_6\lesssim\eta_\nu,
\\
 f(t)&\le C_B a(t)^{1/2}+C_B\eta_\nu^4\mathcal M^+f(t),\\
 \|f\|_{L^2(0,T_j)}
 &\le
 C_B\mathfrak K_\nu([0,T_j])^{1/2}
 +C_B\eta_\nu^4C_M\|f\|_{L^2(0,T_j)}.
\end{align*}
Choose $\nu_0>0$ so that $C_B\eta_\nu^4C_M\le1/2$ whenever
$0<\nu<\nu_0$. By \eqref{dyadicfull} and \eqref{selected-maximal-L2}, for
every such $\nu$ and every $j$,
\begin{align*}
 \|f\|_{L^2(0,T_j)}&\le2C_B\mathfrak K_\nu([0,T_j])^{1/2},\\
 \|\mathcal Y\|_{L^2(0,T_j)}
 &\le C_M\|f\|_{L^2(0,T_j)},\\
 \left\|\begin{pmatrix}f\\\mathcal Y\end{pmatrix}\right\|_{L^2(0,T_j;\R^2)}^2
 &=\int_0^{T_j}(f^2+\mathcal Y^2)
 \le4C_B^2(1+C_M^2)\mathfrak K_\nu([0,T_j]).
\end{align*}
All norms before absorption are finite: $f,\mathcal Y,\mathcal M^+f$
are bounded and measurable, and $h$ is a bounded continuous $L^4$ curve.
The displayed vector is therefore an actual Bochner $L^2$ function.
\end{proof}
The comparison estimate fixes the terminal sequence and the low-frequency control used from this point onward. We now turn to the spatial positivity mechanism.

\section{One-Scale Positivity and a Dynamic Scale}
\subsection{A directional inequality for monotone weights}
For real $v\in H^1(\R_s\times\R_y^2)$ and $a>0$, define
\[
 V:=\int|v|^6,\quad A_\perp:=\int|\nabla_yv|^2,\quad
 A_s:=\int|\partial_sv|^2,\qquad v_a(s,y):=v(as,y).
\]
The weak chain rule \cite{Hunter} and linear
change of variables \cite{Hunter} give
\begin{align*}
 \int|v_a|^6&=a^{-1}V,&
 \int|\nabla_yv_a|^2&=a^{-1}A_\perp,&
 \int|\partial_sv_a|^2&=aA_s.
\end{align*}
Applying \eqref{sharp} to $v_a$ and writing $q:=a^2$ gives
\[
 G^2qV\le(A_\perp+qA_s)^3\qquad(q>0).
\]
For $A_\perp,A_s>0$, choose $q=A_\perp/(2A_s)$; for a zero energy, first
replace both energies on the right by $A_\perp+\epsilon,A_s+\epsilon$,
optimize, and let $\epsilon\downarrow0$. In both cases,
\begin{equation}\label{aniso}
 \int_{\R\times\R^2}|v|^6
 \le\frac{27}{4G^2}\left(\int|\nabla_yv|^2\right)^2
                         \int|\partial_sv|^2.
\end{equation}
For $z\ne0$, write $\widehat z:=z/|z|$; thus $P_z$ below is the
orthogonal projection onto $z^\perp$. At the null diagonal we use
$\widehat0:=0$ and $P_0:=I$.
The anisotropic bound is converted into a two-point tangential estimate by the following lemma.
\begin{lemma}[Monotone angular weights]\label{angularlemma}
Let $w\ge0$ and suppose $b(r)=rw(r)$ is bounded and nondecreasing. For nonnegative $g\in H^1(\R^3)$, define
\[
 A_g=\int|\nabla g|^2,\quad
 X_w(g)=\iint g(y)^2w(|x-y|)|P_{x-y}\nabla g(x)|^2dxdy,
\]
\[
 N_w(g)=\iint g(y)^2w(|x-y|)g(x)^6dxdy,
 \qquad P_z=I-\widehat z\otimes\widehat z.
\]
Then
\begin{equation}\label{angularbound}
 N_w(g)\le\frac{243}{4}\frac{A_g^2}{G^2}X_w(g).
\end{equation}
In particular, $A_g\le\sqrt2G/9$ implies $4X_w-(8/3)N_w\ge0$.
\end{lemma}
\begin{proof}
\emph{1. Exterior disks and monotone weights.} We first prove, for every
bounded nonnegative nondecreasing $b$ and real $v\in H^1(\R\times\R^2)$,
\begin{equation}\label{weightedaniso}
 \int b(|y|)|v|^6
 \le\frac{54}{G^2}\left(\int|\nabla_yv|^2\right)^2
                    \int b(|y|)|\partial_sv|^2.
\end{equation}
Start with $v\in C_c^\infty$ and $R>0$. Set
\[
 v_R(s,y):=
 \begin{cases}
 v(s,R^2y/|y|^2),&0<|y|<R,\\
 0,&y=0,\\
 v(s,y),&|y|\ge R.
 \end{cases}
\]
Choose $M>0$ with $\supp v\subset\{|s|\le M,\ |y|^2\le M\}$.
For $0<|y|^2<R^4/M$,
\[
 |R^2y/|y|^2|^2>M,
 \qquad
 v_R(s,y)=0.
\]
On $|y|=R$,
\[
 R^2y/|y|^2=y,
 \qquad
 \nabla_{s,y}(|y|^2-R^2)=(0,2y)\ne0.
\]
Thus the extension is smooth near the axis, continuous across the circle,
and compactly supported. Sobolev gluing \cite{Schoeberl}
gives $v_R\in H^1$ with the piecewise classical weak derivatives.
The inversion Jacobian and chain rule imply
\begin{align*}
 \int|\nabla_yv_R|^2
 &=2\int_{|y|>R}|\nabla_yv|^2,\\
 \int|\partial_sv_R|^2
 &=\int_{|y|>R}(1+R^4/|y|^4)|\partial_sv|^2
 \le2\int_{|y|>R}|\partial_sv|^2,\\
 \int_{|y|>R}|v|^6
 &\le\int|v_R|^6
 \le
 \frac{27}{4G^2}\left(2\int_{|y|>R}|\nabla_yv|^2\right)^2
                      2\int_{|y|>R}|\partial_sv|^2
\\
 &\le\frac{54}{G^2}\left(\int|\nabla_yv|^2\right)^2
                      \int_{|y|>R}|\partial_sv|^2.
\end{align*}
For the general weight, set
\[
 B(t):=b(\sqrt{\max\{t,0\}}),\qquad B_+:=B(\cdot+),\qquad
 \nu:=dB_+|_{(0,\infty)}.
\]
Monotonicity and boundedness give a finite positive Stieltjes measure and
\begin{gather*}
 B_+(t)=b(0+)+\nu((0,t])\quad(t\ge0),\\
 b(|y|)=b(0+)+\int_{(0,\infty)}\mathbf1_{\{|y|^2\ge t\}}\,d\nu(t)
       \quad\text{a.e. in }(s,y).
\end{gather*}
Indeed $B=B_+$ except at countably many points, and every cylinder
$\{|y|^2=t\}$ is null. Applying Tonelli to the exterior-disk inequality
and to the nonnegative constant $b(0+)$ proves \eqref{weightedaniso} for
$C_c^\infty$ functions. For arbitrary real $v\in H^1$, choose
\begin{gather*}
 v_n\in C_c^\infty,\qquad v_n\to v\text{ in }H^1\cap L^6
 \quad\text{\cite{Hunter}},\\
 v_n^6\to v^6\text{ in }L^1,\qquad
 (\partial_jv_n)^2\to(\partial_jv)^2\text{ in }L^1.
\end{gather*}
Since $b\in L^\infty$, these $L^1$ limits pass \eqref{weightedaniso}
to the original $v$.
Translation gives the same estimate with $b(|y-y'|)$ for each $y'\in\R^2$.

\emph{2. A fixed-direction pair estimate.} First let $g\in C_c^\infty(\R^3;\R)$.
Fix $\omega\in S^2$ and an orthonormal frame $(\omega,e_1,e_2)$. Define
\begin{gather*}
 Q_\omega(s,y):=s\omega+y_1e_1+y_2e_2,\qquad g_\omega:=g\circ Q_\omega,\\
 \partial_sg_\omega=(\omega\cdot\nabla g)\circ Q_\omega,\qquad
 \partial_{y_j}g_\omega=(e_j\cdot\nabla g)\circ Q_\omega\quad(j=1,2).
\end{gather*}
Until \eqref{planepair}, write $g$ for this pullback and put
\begin{gather*}
 d:=\partial_sg,\qquad\beta(y,y'):=b(|y-y'|)=\beta(y',y),\\
 I_b:=\int\beta(y,y')g(s,y')^2d(s,y)^2\,ds\,dy\,dy',\\
 N_b:=\int\beta(y,y')g(s,y')^2g(s,y)^6\,ds\,dy\,dy',\\
 Y:=\pi_y(\supp g)\cup\pi_y(\supp d),\qquad |Y|<\infty,\\
 y\notin Y\qquad g(s,y)=d(s,y)=0.
\end{gather*}
Choose $M>0$ with $g^2\le M$. For $y'\in Y$ and $\epsilon>0$, define
\begin{gather*}
 h(s):=g(s,y'),\quad H(s):=h(s)^2+\epsilon,\quad
 a(s):=H(s)^{1/4},\quad\tau(s):=\int_0^sH(t)^{-1/2}\,dt,\\
 \tau'=a^{-2}\ge(M+\epsilon)^{-1/2}>0,\qquad
 \tau(t)-\tau(s)\ge(M+\epsilon)^{-1/2}(t-s)\quad(s\le t).
\end{gather*}
Hence $\tau:\R\to\R$ is a smooth increasing diffeomorphism. Define
$v(\tau(s),y):=a(s)g(s,y)$. Then $v\in C_c^\infty$ and
\begin{gather*}
 a'=hh'/(2a^3),\qquad
 \nabla_yv(\tau(s),y)=a\nabla_yg(s,y),\\
 v_\tau(\tau(s),y)=a^2(ag_s+a'g),\qquad d\tau=a^{-2}ds,\\
 \int|\nabla_yv|^2d\tau dy=\int|\nabla_yg|^2=:A_\perp,\\
 \int\beta v^6d\tau dy=\int\beta Hg^6dsdy.
\end{gather*}
For real $q,q_s,p,p_s$, put
\begin{align*}
 R_\epsilon(q,q_s,p,p_s)
 &:=(q^2+\epsilon)p_s^2+qq_spp_s
       +\frac{q^2q_s^2p^2}{4(q^2+\epsilon)}\\
 &=\left(\sqrt{q^2+\epsilon}\,p_s+
          \frac{qq_sp}{2\sqrt{q^2+\epsilon}}\right)^2\ge0,\\
 \frac{q^2}{q^2+\epsilon}\le1,\quad
 2qq_spp_s\le q^2p_s^2+p^2q_s^2
 &\qquad
 R_\epsilon\le\tfrac32q^2p_s^2+\tfrac34p^2q_s^2+\epsilon p_s^2.
\end{align*}
The exact transformed energy and its symmetric bound are
\begin{align*}
 \int\beta|v_\tau|^2d\tau dy
 &=\int\beta R_\epsilon(h,h',g,g_s)dsdy,\\
 \int_{\R\times Y\times Y}\beta R_\epsilon(h,h',g,g_s)
 &\le\left(\tfrac32+\tfrac34\right)I_b+\epsilon C_Y
 =\tfrac94I_b+\epsilon C_Y,\\
 C_Y&:=\int_{\R\times Y\times Y}\beta|g_s(s,y)|^2<\infty.
\end{align*}
The second line uses $(y,y')\leftrightarrow(y',y)$. Boundedness of
$\beta,g$, finite $|Y|$, and $g^6,g_s^2\in L^1$ justify Fubini and every
integral above. Apply \eqref{weightedaniso} at each fixed center $y'$:
\begin{align*}
 N_b&\le\int_Y\int\beta Hg^6\,dsdy\,dy'\\
 &\le\frac{54}{G^2}A_\perp^2
       \int_{\R\times Y\times Y}\beta R_\epsilon(h,h',g,g_s)\\
 &\le\frac{54}{G^2}A_\perp^2\left(\tfrac94I_b+\epsilon C_Y\right).
\end{align*}
Letting $\epsilon\downarrow0$ gives
\begin{equation}\label{planepair}
 N_b\le\frac{243}{2G^2}A_\perp^2I_b.
\end{equation}
The $I_b,N_b$ integrands vanish off $\R\times Y\times Y$; the error
$\epsilon C_Y$ was integrated only on that finite transverse set.

\emph{3. Incidence measure and its normalization.} Return to physical $g$.
Orthonormality gives
\[
 A_{\perp,\omega}
 =\int_{\R^3}\bigl(|\nabla g|^2-|\omega\cdot\nabla g|^2\bigr)\le A_g.
\]
For $z\ne0$, choose an orthonormal frame $(\widehat z,a_z,c_z)$ and set
\begin{gather*}
 C_z:=\{\omega\in S^2:\omega\cdot z=0\},\qquad
 \gamma_z(\theta):=a_z\cos\theta+c_z\sin\theta,\\
 \nu_z:=|z|^{-1}(\gamma_z)_\#(\mathbf1_{(0,2\pi]}d\theta),\qquad\nu_0:=0.
\end{gather*}
Here $\gamma_z:(0,2\pi]\to C_z$ is bijective and $|\gamma'_z|=1$.
Frame changes are phase shifts or reversals of $\theta$, so $\nu_z$ is
$|z|^{-1}$ times intrinsic arclength, independently of the frame.
For measurability, fix a unit $e_0$ and a frame at $e_0$, and use
\[
 H_nx:=x-2\frac{(e_0-n)\cdot x}{|e_0-n|^2}(e_0-n),\qquad
 H_{e_0}:=I,\qquad n:=\widehat z.
\]
This Borel Householder frame makes $z\mapsto\nu_z$ a Borel measure kernel.
For $z\ne0$ and $\xi,\eta\in\R^3$, direct circle integration gives
\begin{equation}\label{great-circle-moments}
 \begin{aligned}
 \int d\nu_z&=2\pi/|z|,\\
 \int(\omega\cdot\xi)(\omega\cdot\eta)\,d\nu_z
 &=\frac\pi{|z|}\bigl[\xi\cdot\eta
       -(\widehat z\cdot\xi)(\widehat z\cdot\eta)\bigr],\\
 \int|\omega\cdot\xi|^2d\nu_z&=\pi|P_z\xi|^2/|z|.
 \end{aligned}
\end{equation}
These are measure integrals. Let $d\sigma$ be unnormalized area on $S^2$.
For every nonnegative measurable $F$,
\begin{equation}\label{incidence-coarea}
 \begin{aligned}
 &\int_{S^2}\int_\R\int_{\omega^\perp}\int_{\omega^\perp}
 F(\omega,s\omega+y,s\omega+y')\,dy'\,dy\,ds\,d\sigma(\omega)\\
 &\qquad=\iint_{\R^3\times\R^3}
 \int_{S^2}F(\omega,x,x')\,d\nu_{x-x'}(\omega)\,dx\,dx'.
 \end{aligned}
\end{equation}
The successive changes of variables are
\begin{gather*}
 x=s\omega+y,\quad z=y-y'\in\omega^\perp,\quad x'=x-z,
 \qquad dz|_{\omega^\perp}=r\,dr\,d\ell_{C_\omega}(u),\quad z=ru,\\
 d\sigma(\omega)d\ell_{C_\omega}(u)
 \longleftrightarrow d\sigma(u)d\ell_{C_u}(\omega)
 \quad\text{\cite{SchneiderWeil}},\\
 r\,dr\,d\sigma(u)=|z|^{-1}dz.
\end{gather*}
Tonelli and the invariant flag formula give \eqref{incidence-coarea}.
The null diagonal makes the convention $\nu_0=0$ immaterial.

\emph{4. Averaging the fixed-direction estimate.} Use
\begin{align*}
 F_N(\omega,x,x')&:=b(|x-x'|)g(x')^2g(x)^6,\\
 F_X(\omega,x,x')&:=b(|x-x'|)g(x')^2|\omega\cdot\nabla g(x)|^2.
\end{align*}
For smooth compactly supported $g$, bounded $b$ and
$|z|^{-1}\in L^1_{\rm loc}(\R^3)$ imply absolute integrability after the
shear $(x,x')\mapsto(x-x',x')$; also $|P_z\nabla g|\le|\nabla g|$.
Thus \eqref{incidence-coarea}, \eqref{great-circle-moments}, and $b(r)/r=w(r)$
give
\begin{align*}
 \int_{S^2}N_{b,\omega}\,d\sigma&=2\pi N_w(g),\\
 \int_{S^2}I_{b,\omega}\,d\sigma&=\pi X_w(g),\\
 2\pi N_w(g)
 &\le
 \frac{243}{2G^2}\int_{S^2}A_{\perp,\omega}^2I_{b,\omega}\,d\sigma
\\
 &\le\frac{243}{2G^2}A_g^2\pi X_w(g).
\end{align*}
Dividing by $2\pi$ gives \eqref{angularbound}.

\emph{5. Passage to actual weak $H^1$ data.} Choose one real sequence
$g_n\in C_c^\infty$ with $g_n\to g$ in $H^1\cap L^6$
\cite{Hunter}. Set $B_b:=\|b\|_\infty$.
For either scalar or matrix-entry kernel
\begin{align*}
 K(x,y)&=w(|x-y|),\quad\text{or}\\
 K(x,y)&=w(|x-y|)(P_{x-y})_{ij},\\
 |K(x,y)|&\le B_b/|x-y|\qquad(x\ne y).
\end{align*}
translated Hardy \cite{DietzeNam} yields
\begin{align*}
 \sup_x\left|\int K(x,y)(g_n(y)^2-g(y)^2)\,dy\right|
 &\le2B_b\|g_n-g\|_2
        (\|\nabla g_n\|_2+\|\nabla g\|_2)\longrightarrow0,\\
 g_n^6&\longrightarrow g^6\quad\text{in }L^1,\\
 (\partial_i g_n)(\partial_j g_n)
 &\longrightarrow(\partial_i g)(\partial_j g)\quad\text{in }L^1.
\end{align*}
The same Hardy bound controls the potentials uniformly. Therefore
\begin{gather*}
 N_w(g_n)\longrightarrow N_w(g),\qquad
 X_w(g_n)\longrightarrow X_w(g),\qquad A_{g_n}\longrightarrow A_g.
\end{gather*}
Passing to the limit in \eqref{angularbound} for $g_n$ proves the same
inequality for $g$. Both limiting pairings are absolutely integrable. Finally,
\[
 A_g\le\sqrt2G/9
 \qquad\tfrac{243}{4}A_g^2/G^2\le\tfrac32
 \qquad N_w\le\tfrac32X_w
 \qquad4X_w-\tfrac83N_w\ge0.
\]
\end{proof}
This finishes the angular estimate for every admissible radial weight. We next specialize the geometry to the Gaussian family used in the interaction functional.

\subsection{Gaussian interaction geometry}
Define the single-scale family
\begin{equation}\label{gaussiankernel}
 \Phi(s)=e^{-s^2/2},\quad
 K_0(s)=s\int_0^s\Phi(t)dt-1+\Phi(s),\qquad
 k_R(x)=R K_0(|x|/R).
\end{equation}
This family is defined at a single scale and contains no radius averaging. Set
\[
 \Phi_R(r)=R^{-1}\Phi(r/R),\quad
 \Xi_R(r)=r^{-1}\int_0^r\Phi_R(t)dt\quad(r>0),
 \quad w_R=\Xi_R-\Phi_R,
\]
The continuous values at the origin are $\Xi_R(0):=R^{-1}$ and $w_R(0):=0$.
\begin{equation*}
 H_R(z)=D^2k_R(z)=\Phi_R I+w_RP_z,\qquad
 L_R=\Delta k_R=\Phi_R+2\Xi_R.
\end{equation*}
One has $K_0''=\Phi$, $K_0-sK_0'=\Phi-1$, and $rw_R(r)=\int_0^r-t\Phi_R'(t)dt$ is bounded and nondecreasing. Thus Lemma~\ref{angularlemma} applies uniformly in $R$.

For a radial $\Phi$ with Fourier transform $\Gamma$, radial averaging gives, for $r>0$,
\[
 \widehat{\Xi_1}(r)=\int_0^1t^{-3}\Gamma(r/t)dt=r^{-2}\int_r^\infty s\Gamma(s)ds.
\]
For the Gaussian, $\Gamma(\xi)=(2\pi)^{3/2}e^{-|\xi|^2/2}$, so
\begin{equation}\label{densityFourier}
 \Q_R(\rho):=-\int\rho\Delta(L_R*\rho)
 =(2\pi)^{-3}\int(R^2|\xi|^2+2)\Gamma(R\xi)|\widehat\rho(\xi)|^2d\xi\ge0.
\end{equation}
Also $Q_{\Phi,R}(\rho):=-\int\rho\Delta(\Phi_R*\rho)\ge0$ and $Q_{\Phi,R}\le\Q_R$. All these are statements about complete Fourier quadratic forms. The physical density kernel $R^{-3}(5-|x|^2/R^2)\Phi(|x|/R)$ changes sign and is never discarded pointwise.
The second spatial ingredient controls the full convolution quadratic form without discarding this sign-changing kernel.

\begin{lemma}[Weighted convolution Sobolev]\label{isotropic}
Let $v\in C^2(\R^3)$ be even and nonnegative, with $v$ and all its first
and second derivatives bounded. Let $g\in H^1(\R^3;\R)$ be nonnegative,
$\rho=g^2$, $U=v*\rho$, and put
\[
 X_v=\int U|\nabla g|^2,\quad N_v=\int Ug^6,\quad Q_v=-\int\rho\Delta U.
\]
These quantities are finite, $\Delta U=(\Delta v)*\rho$, and
\begin{equation*}
 N_v\le\frac{A_g^2}{G^2}\left(2X_v+\frac12Q_v\right).
\end{equation*}
\end{lemma}
\begin{proof}
\emph{Smooth data and a positive regularization.} First let
$g\in C_c^\infty(\R^3;\R)$, $\rho:=g^2$, $U:=v*\rho$, and
$U_\epsilon:=U+\epsilon$, $\epsilon>0$. Bounded kernel derivatives give
\[
 \nabla U=v*\nabla\rho=2v*(g\nabla g),\qquad
 \Delta U=(\Delta v)*\rho,\qquad U_\epsilon>0.
\]
Pointwise Cauchy--Schwarz and the evenness of $v$ imply
\begin{align*}
 |\nabla U(x)|^2
 &\le4U(x)\int v(x-y)|\nabla g(y)|^2\,dy,\\
 \mathcal F_\epsilon:=\int\rho\frac{|\nabla U|^2}{U_\epsilon}
 &\le4\iint\rho(x)v(x-y)|\nabla g(y)|^2\,dxdy=4X_v.
\end{align*}
Every denominator is positive. H\"older and two uses of \eqref{sharp}
\cite{Talenti} give
\begin{align*}
 N_v+\epsilon\int g^6
 &=\int U_\epsilon g^6\\
 &\le\left(\int g^6\right)^{2/3}
        \left(\int U_\epsilon^3g^6\right)^{1/3}\\
 &\le\frac{A_g^2}{G^2}
        \int|\nabla(U_\epsilon^{1/2}g)|^2.
\end{align*}
Expansion and integration by parts, justified by compact support, yield
\begin{align*}
 \int|\nabla(U_\epsilon^{1/2}g)|^2
 &=X_v+\epsilon A_g
       +\tfrac12\int\nabla\rho\cdot\nabla U+\tfrac14\mathcal F_\epsilon\\
 &=X_v+\epsilon A_g+\tfrac12Q_v+\tfrac14\mathcal F_\epsilon\\
 &\le2X_v+\epsilon A_g+\tfrac12Q_v.
\end{align*}
Letting $\epsilon\downarrow0$ proves the required inequality for smooth,
possibly signed $g$, without any sign condition on $Q_v$.

\emph{Weak $H^1$ data.} Choose one real $g_n\in C_c^\infty$ with
$g_n\to g$ in $H^1\cap L^6$
\cite{Hunter}, and set
$U_n:=v*g_n^2$. Then
\begin{gather*}
 g_n^2\to g^2,\qquad g_n^6\to g^6,\qquad
 |\nabla g_n|^2\to|\nabla g|^2\quad\text{in }L^1,\\
 \|U_n-U\|_\infty\le\|v\|_\infty\|g_n^2-g^2\|_1\longrightarrow0,\\
 \|\Delta U_n-\Delta U\|_\infty
 \le\|\Delta v\|_\infty\|g_n^2-g^2\|_1\longrightarrow0.
\end{gather*}
Boundedness of $U_n,\Delta U_n$ and the unweighted $L^1$ limits imply
\[
 X_v(g_n)\to X_v(g),\qquad N_v(g_n)\to N_v(g),\qquad
 Q_v(g_n)\to Q_v(g),\qquad A_{g_n}\to A_g.
\]
All limiting integrals are finite; passing to the limit proves the claim.
\end{proof}
Together with the angular lemma, this estimate gives the Gaussian lower bound needed for uniform coercivity.
For $v=\Phi_R$, combine Lemmas~\ref{angularlemma} and \ref{isotropic}. If $P_R(g;g^2)$ denotes the kinetic and focusing parts of $\mathfrak B_{k_R}[g]$, and $a_g=A_g^2/G^2$, then
\begin{equation}\label{tailsmall}
 P_R(g;g^2)\ge(4-8a_g)X_{\Phi_R}(g)+(4-162a_g)X_{w_R}(g)
                   -2a_gQ_{\Phi,R}(g^2).
\end{equation}
In particular, for $A_g\le G/9$ both kinetic coefficients are nonnegative. The absolute threshold $G/9$ is smaller than the classical $4G/9$ threshold, but the tail below can be chosen below either fixed positive number.

\subsection{Uniform coercivity near the compact orbit}
For nonnegative $c\in H^1$, define
\[
 T_c:=\int\nabla c\otimes\nabla c,\qquad A_c:=\tr T_c,\qquad
 B_c:=\int c^6.
\]
If $A_c<G$ and $E(c)<G/3$, then
\begin{equation}\label{tensor}
 T_c-\frac{B_c}{3}I\succeq
 \frac{2(G/3-E(c))}{G-A_c}\,T_c.
\end{equation}
Indeed, for a unit vector $\omega$, put
$D:=\|\partial_\omega c\|_2^2$ and $T:=A_c-D$. Linear pullback
\cite{Hunter} and
\eqref{sharp} \cite{Talenti} give
\begin{align*}
 B_cG^2q&\le(T+qD)^3\qquad(q>0).
\end{align*}
If $D>0$, choose
\[
 q:=\frac{G-T}{D}=1+\frac{G-A_c}{D}>0.
\]
Then
\begin{align*}
 B_c(G-T)&\le GD,\\
 GD-B_c(G-T)&\ge0,\\
 D-\frac{B_c}{3}
 &\ge\frac{2(G/3-E(c))}{G-A_c}D.
\end{align*}
If $D=0$, then $B_c=0$.
The last two lines cover every direction and prove \eqref{tensor}.
For weak $H^1$ data, one smooth sequence converges in $H^1\cap L^6$
\cite{Hunter}; each directional
quadratic moment is a finite linear combination of $L^1$ gradient products.
Thus the same inequalities pass to the limit in every frame.

The Gaussian kernel satisfies
\begin{align}
 &\|H_R\|_\infty+\|L_R\|_\infty\le C/R,
 \quad H_R(z)=R^{-1}I+O(|z|^2/R^3),\label{kernelNear}\\
 &|\nabla H_R(z)|+|\nabla L_R(z)|\le C L_R(z)/|z|,
 \qquad\|\Delta\Phi_R\|_\infty\le C/R^3.\label{kernelRelative}
\end{align}
These follow by scaling and $s\Xi'(s)=\Phi(s)-\Xi(s)$. For $s\le1$, use $\Xi\ge\Phi$; for $s\ge1$, use $\Xi(s)\ge s^{-1}\int_0^1\Phi$ to bound $s^2\Phi/\Xi$. Thus the relative derivative ratios are bounded at all scales.
The preceding Gaussian identities can now be made uniform on a homogeneous neighborhood of the compact orbit.

\begin{proposition}[Gaussian core gap]\label{spatial}
Let $\K\subset\dot H^1$ be compact, not containing zero, with strict uniform energy and gradient thresholds. There are $\epsilon_0,R_0,c>0$ such that every $h\in H^1$ within $\epsilon_0$ of a translate of $\K$ in $\dot H^1$ satisfies, with
\[
 \mathfrak C_R(h):=R\mathfrak B_{k_R}[h]-\frac{77}{81}R\Q_R(|h|^2),
\]
\begin{equation}\label{spatialbound}
 \inf_{R\ge R_0}\mathfrak C_R(h)\ge c.
\end{equation}
No bound on $\|h\|_2$ is used in these constants.
\end{proposition}
\begin{proof}
\emph{1. Constants before cutoffs.} Center $h$ and set $f:=|h|$.
Polar factorization \cite{ZhengPolar} gives
$A_f\le A_h$ and $E(f)\le E(h)$. Compactness, the strict gaps and
$0\notin\K$ give fixed $a,e,b>0$ on a sufficiently small homogeneous
neighborhood:
\[
 A_h\le G-a,\qquad E(h)\le G/3-e,\qquad \int f^6\ge b.
\]
Fix also an annulus $0<r_-<r_+$ carrying positive weighted mass, as
justified below, and a smooth cutoff with derivative constant $C_0$.
Put $\sigma:=4eb/(27G)>0$. Choose $\eta>0$ jointly so that
\begin{equation}\label{core-smallness-choice}
 \begin{gathered}
 4C_0^2\eta^2\le a/2,\qquad
 2C_0^2\eta^2+\eta^6/6\le e/2,\qquad \eta^6\le b/2,\\
 (6+8C_0^2)\eta^2\le G/9,\qquad
 16C_0^2\eta^2+\tfrac43\eta^6\le\sigma,\qquad
 4\eta^6\le\sigma/4.
 \end{gathered}
\end{equation}
Such a choice exists because every left-hand side tends to zero.
The bounded maps
\[
 \dot H^1\longrightarrow L^6,\quad v\longmapsto v;\qquad
 \dot H^1\longrightarrow L^2,\quad v\longmapsto v/|x|
\]
are Sobolev \cite{Tao} and Hardy
\cite{DietzeNam}. Their compact images have
uniform tails. Shrink the same neighborhood and choose $C\ge r_+$ so that
\begin{equation}\label{core-three-tails}
 \int_{|x|>C/2}|\nabla f|^2\le3\eta^2,\qquad
 \int_{|x|>C/2}\frac{f^2}{|x|^2}\le\eta^2,\qquad
 \int_{|x|>C/2}f^6\le\eta^6.
\end{equation}

\emph{2. Exact weak cutoff identities.} Choose
\begin{gather*}
 \theta=0\text{ on }B_C,\quad \theta=\pi/2\text{ outside }B_{2C},\quad
 0\le\theta\le\pi/2,\quad |\nabla\theta|\le C_0/C,\\
 \alpha:=\cos\theta,\quad\beta:=\sin\theta,\quad
 c:=\alpha f,\quad g:=\beta f,\quad\rho:=f^2=c^2+g^2.
\end{gather*}
The weak product rule \cite{Hunter} gives
\begin{align*}
 \nabla f\otimes\nabla f
 &=\nabla c\otimes\nabla c+\nabla g\otimes\nabla g
       -f^2\nabla\theta\otimes\nabla\theta,\\
 e_6:=f^6-c^6-g^6&=3\alpha^2\beta^2f^6\ge0.
\end{align*}
Consequently \eqref{core-three-tails} and \eqref{core-smallness-choice} yield
\begin{align*}
 \int f^2|\nabla\theta|^2&\le4C_0^2\eta^2,\qquad
 \int e_6\le\eta^6,\qquad\int g^6\le\eta^6,\\
 A_g&\le(6+8C_0^2)\eta^2\le G/9,\\
 A_c&\le A_f+4C_0^2\eta^2\le G-a/2,\\
 E(c)&\le E(f)+2C_0^2\eta^2+\eta^6/6\le G/3-e/2,\\
 B_c&\ge\int f^6-\eta^6\ge b/2,\qquad
 B_c\le G.
\end{align*}
For every unit $\omega$, \eqref{aniso} and $A_c\le G$ give the
first bound below. Applying \eqref{tensor} with $G/3-E(c)\ge e/2$
then gives the second:
\[
 \|\partial_\omega c\|_2^2\ge4B_c/27,\qquad
 4T_c-\tfrac43B_cI\succeq2\sigma I.
\]
Define the core stress and its nonnegative rank-one majorant by
\begin{align*}
 S_c&:=4\nabla c\otimes\nabla c-\tfrac43c^6I
       -4f^2\nabla\theta\otimes\nabla\theta-\tfrac43e_6I,\\
 \mathfrak s_c&:=4|\nabla c|^2+4f^2|\nabla\theta|^2+4(c^6+e_6),\\
 M_0&:=8G+16C_0^2\eta^2+4\eta^6.
\end{align*}
For symmetric matrices $A$, $|A:S_c|\le\|A\|_{\rm op}\mathfrak s_c$.
The cutoff error in every unit direction is at most
$16C_0^2\eta^2+\tfrac43\eta^6\le\sigma$. Hence
\[
 \int S_c\succeq\sigma I,\qquad
 \int\mathfrak s_c\le M_0,\qquad
 \operatorname{ess\,supp}S_c\subset\overline B_{2C}.
\]
The constants $M_0,\sigma$ do not depend on the chosen radius $C$;
the fields $S_c,\mathfrak s_c$ use that radius.

\emph{3. A positive contribution at every density point.} For $R>0$, define
$F_R(y):=\int H_R(x-y):S_c(x)\,dx$. The eigenvalues of $H_R(-y)$ are
$\Phi_R,\Xi_R,\Xi_R\ge0$, including their continuous values at zero. Thus
\[
 H_R(-y):\int S_c\ge\sigma L_R(y).
\]
The relative bounds \eqref{kernelRelative} and segment comparison imply,
for $|y|\ge4C$,
\[
 \left|F_R(y)-H_R(-y):\int S_c\right|
 \le\frac{144C}{|y|}L_R(y)M_0.
\]
Set $D:=\max\{4,288M_0/\sigma\}$. Then, for $|y|\ge DC$,
\[
 F_R(y)\ge\tfrac12\sigma L_R(y)\ge0.
\]
In the near region, the quadratic remainder in \eqref{kernelNear} gives
\[
 \left|F_R(y)-R^{-1}\tr\int S_c\right|
 \le\frac{(2C+|y|)^2M_0}{2R^3}.
\]
If $R\ge R_0\ge1+(D+2)^2C^2M_0/\sigma$ and $|y|\le DC$, then
\[
 F_R(y)\ge\sigma/R.
\]
The Hessian is continuous at zero. Dominated convergence and
$\|H_R\|_{\rm op}\le R^{-1}$ give
\[
 F_R\in C(\R^3),\qquad |F_R|\le M_0/R,\qquad
 \rho F_R\in L^1,\qquad F_R\ge0.
\]
Put $c_1:=\sigma$, $M_c:=\int c^2$, and $C_2:=16G$. Since $c^2\le\rho$ and
$\supp c\subset\overline B_{2C}$,
\[
 \int\rho F_R\ge c_1M_c/R,\qquad
 0<m_0\le M_c\le C_2C^2.
\]
For the lower bound, the compact image $\{v/|x|:v\in\K\}$ excludes zero;
there exist the already fixed $r_-,r_+$ and $m_{\rm w}>0$ such that
\[
 \inf_{v\in\K}\int_{r_-<|x|\le r_+}\frac{|v(x)|^2}{|x|^2}\,dx
 \ge m_{\rm w}>0.
\]
After shrinking the neighborhood,
\[
 M_c\ge\int_{r_-<|x|\le r_+}f^2
 \ge r_-^2m_{\rm w}/2=:m_0>0.
\]
The upper bound is local Hardy:
$M_c\le\int_{B_{2C}}|h|^2\le16C^2\|\nabla h\|_2^2\le16GC^2$.
Neither bound uses the total mass of $h$.

\emph{4. Keep the complete positive density form.} For $\mu\ge0$, set
\begin{align*}
 P_R(g;\mu)&:=\iint\mu(y)H_R(x-y):
       (4\nabla g\otimes\nabla g-\tfrac43g^6I)(x)\,dxdy,\\
 a_g&:=A_g^2/G^2\le1/81.
\end{align*}
Using \eqref{tailsmall}, the exact decomposition and the Fourier seminorm triangle inequality give
\begin{align*}
 \mathfrak B_{k_R}[f]
 &=\Q_R(\rho)+\int\rho F_R+P_R(g;g^2)+P_R(g;c^2),\\
 P_R(g;g^2)&\ge-2a_gQ_{\Phi,R}(g^2),
\\
 Q_{\Phi,R}(g^2)&=Q_{\Phi,R}(\rho-c^2)
 \le2\Q_R(\rho)+2Q_{\Phi,R}(c^2),\\
 0\le Q_{\Phi,R}(c^2)&\le5M_c^2/R^3,\\
 P_R(g;c^2)&\ge-4M_cR^{-1}\|g\|_6^6
 \ge-4\eta^6M_c/R.
\end{align*}
Here $Q_{\Phi,R}^{1/2}$ is a Fourier Hilbert seminorm; no separated density
cross term is assigned a sign. By \eqref{core-smallness-choice}, $4\eta^6\le\sigma/4$.
Enlarge $R_0$ so that
\[
 R_0\ge1+1280GC^2/(81\sigma),
 \qquad
 20a_gM_c/R^2\le\sigma/4\quad(R\ge R_0).
\]
For every $R\ge R_0$, all preceding bounds combine as
\begin{align*}
 \mathfrak B_{k_R}[f]
 &\ge(1-4a_g)\Q_R(\rho)
       +\frac{M_c}{R}\left(\sigma-4\eta^6-\frac{20a_gM_c}{R^2}\right)\\
 &\ge\frac{77}{81}\Q_R(\rho)+\frac{\sigma M_c}{2R}
 \ge\frac{77}{81}\Q_R(\rho)+\frac{\sigma m_0}{2R}.
\end{align*}
The bounded even positive Gaussian Hessian and $h\in H^1$ satisfy the
integrability hypotheses of Lemma~\ref{phase}. Therefore
\[
 \mathfrak B_{k_R}[h]\ge\mathfrak B_{k_R}[|h|],\qquad
 \Q_R(|h|^2)=\Q_R(f^2),\qquad
 \inf_{R\ge R_0}\mathfrak C_R(h)\ge\sigma m_0/2>0.
\]
Translations preserve the weak gradient, homogeneous distance and complete
pairings by Haar change of variables. The uniform choices have the order
\[
 (a,e,b,m_0,r_\pm,C_0)\ \longrightarrow\ \eta
 \ \longrightarrow\ (\epsilon_0,C)
 \ \longrightarrow\ D\ \longrightarrow\ R_0.
\]
No selected-root parameter enters this argument.
\end{proof}
The proposition supplies a radius-uniform positive core gap. It remains to choose, at each time, a single Gaussian radius compatible with that gap.

\subsection{Gaussian scale selection}
A single Gaussian scale is selected from the projected density.
For $s>0$, set $\Psi_s(x):=\Phi(|x|/s)=\exp(-|x|^2/(2s^2))=s\Phi_s(|x|)$.
Time-dependent Morawetz rescaling and scale averaging have earlier
precedents in \cite{DMMZ,CH}.
The required static comparability and moving-scale error are contained in the next lemma.

\begin{lemma*}[Gaussian selector]
Let \(h\in H^1(\mathbb R^3)\setminus\{0\}\),
\(\rho=|h|^2\), and \(a=\|h\|_4^4\), with
\(\|\nabla h\|_2\le B\).  Given \(L>0\), one may fix
\(\vartheta=\vartheta(B,L)>0\) and associate to \(\rho\) a scale
\(\lambda=\Lambda_\vartheta(\rho)\) for which
\begin{align}
 La&\le\lambda\le C_{\vartheta}a,\label{selectorbounds}\\
 \Q_\lambda(\rho)&\ge c_{\vartheta}\lambda.\label{selectorQ}
\end{align}
For the time-dependent assertion, let $I$ be an open time interval and
$h\in C(I;H^1(\R^3)\setminus\{0\})$ with the same uniform gradient bound.
Assume $p\in L^\infty_{\mathrm{loc}}(I;L^1(\R^3;\R^3))$ and
$r\in L^\infty_{\mathrm{loc}}(I;L^1(\R^3))$, and suppose, in distributions,
\[
 \rho_t=-2\divg p+r,
 \qquad
 \mathcal J_\lambda
 :=2\int p\cdot(\nabla\Psi_\lambda*\rho),
\]
Then \(\lambda(t)\in W^{1,\infty}_{\mathrm{loc}}(I)\), and almost everywhere
\begin{equation}\label{radiuserror}
 \lambda'(t)\mathcal J_\lambda
 \ge -C_{\vartheta}\,a(t)\|r(t)\|_1^2 .
\end{equation}
\end{lemma*}

\begin{proof}
\emph{1. The same density root for every assertion.} For $s>0$, set
\begin{align*}
 \mathcal G(s;\rho)&:=\int\rho(\Psi_s*\rho)
 =(2\pi)^{-3}s^3\int\Gamma(s\xi)|\widehat\rho(\xi)|^2\,d\xi,\\
 \mathfrak H_\rho(\sigma)&:=e^{-4\sigma}\mathcal G(e^\sigma;\rho),\qquad
 d\mu_\rho(\xi):=(2\pi)^{-3}|\widehat\rho(\xi)|^2\,d\xi.
\end{align*}
Since $h\in H^1\setminus\{0\}$,
\[
 \rho\in L^1\cap L^2,\qquad \mu_\rho(\R^3)=\|\rho\|_2^2=a>0,\qquad
 s^{-4}\mathcal G(s;\rho)=s^{-1}\int\Gamma(s\xi)\,d\mu_\rho.
\]
Gaussian monotonicity, differentiation and dominated convergence give
\begin{align*}
 s\partial_s\mathcal G(s;\rho)&\le3\mathcal G(s;\rho),\\
 -\partial_\sigma\log\mathfrak H_\rho(\sigma)&\ge1,\\
 \lim_{\sigma\to-\infty}\mathfrak H_\rho(\sigma)&=\infty,\qquad
 \lim_{\sigma\to+\infty}\mathfrak H_\rho(\sigma)=0.
\end{align*}
Thus $\mathfrak H_\rho:\R\to(0,\infty)$ is a strictly decreasing bijection.
For $\vartheta>0$, define
\begin{equation}\label{selector-root-definition}
 \lambda:=\Lambda_\vartheta(\rho)
 :=\exp(\mathfrak H_\rho^{-1}(\vartheta)),\qquad
 \mathcal G(\lambda;\rho)=\vartheta\lambda^4.
\end{equation}

\emph{2. Choose the level before the profile.} Fractional Leibniz
\cite{GrafakosOh}
with pairs $(3,6),(6,3)$, and homogeneous Sobolev
\cite{Tao}, give
\begin{align*}
 \|fg\|_{\dot H^{1/2}}
 &\le C\bigl(\||\nabla|^{1/2}f\|_3\|g\|_6
             +\|f\|_6\||\nabla|^{1/2}g\|_3\bigr)\\
 &\le C\|\nabla f\|_2\|\nabla g\|_2,\qquad f,g\in H^1.
\end{align*}
For $H^1$ data the passage from smooth data uses density
\cite{Hunter}, convergence of products in $L^1$,
and Fatou. With $f=g=|h|$, the modulus bound
\cite{ZhengPolar} yields
\[
 \int|\xi|\,d\mu_\rho(\xi)\le C_*B^4.
\]
Fix the profile-independent constants
\[
 D_B:=2C_*B^4+1,\qquad
 \vartheta:=\frac{\Gamma(0)e^{-(LD_B)^2/2}}{4L}>0.
\]
The frequency moment gives
\begin{align*}
 \mu_\rho\{|\xi|>D_B/a\}
 &\le\frac a{D_B}\int|\xi|\,d\mu_\rho\le a/2,\\
 \mathfrak H_\rho(\log(La))
 &\ge\frac{\Gamma(0)e^{-(LD_B)^2/2}}{La}\,\frac a2
 =2\vartheta\ge\vartheta.
\end{align*}
Since $\mathfrak H_\rho$ is decreasing,
\[
 La\le\lambda.
\]
Moreover,
\begin{align*}
 \vartheta=\lambda^{-1}\int\Gamma(\lambda\xi)\,d\mu_\rho
 &\le\Gamma(0)a/\lambda,\\
 \lambda&\le\Gamma(0)a/\vartheta,\\
 \Q_\lambda(\rho)
 &\ge
 2\int\Gamma(\lambda\xi)\,d\mu_\rho
 =2\vartheta\lambda.
\end{align*}
The last bound uses \eqref{densityFourier}. This proves the static assertions with the same root \eqref{selector-root-definition}.

\emph{3. Weak time differentiation on one common set.} Let
\begin{align*}
 X&:=C_b^1(\R^3;\R),\qquad
 \|\phi\|_X:=\max\{\|\phi\|_\infty,\|\nabla\phi\|_\infty\},\\
 Jf(\phi)&:=\int f\phi,\qquad T_sf:=\Psi_s*f.
\end{align*}
Cutoff and mollification
\cite{Hunter}
extend the distributional identity to $X$. The approximants and their first
derivatives are uniformly bounded, so dominated convergence against
$\rho,p,r\in L^1$ gives the extension. In $X^*$,
\begin{gather*}
 J\rho\in C(I;X^*),\qquad
 \mathcal D_t(\phi):=2\int p(t)\cdot\nabla\phi+\int r(t)\phi,\\
 \|\mathcal D_t\|_{X^*}\le2\|p(t)\|_1+\|r(t)\|_1,\qquad
 (J\rho)'(t)=\mathcal D_t\quad\text{for a.e. }t\in I.
\end{gather*}
The last line is the Banach-valued weak fundamental theorem
\cite{Hunter}. Its exceptional set is common
to all $X$ tests; no strong $L^1$ derivative of $\rho$ is assumed.

\emph{4. Lipschitz continuity of the selected root.} Fix a compact interval
$J\Subset I$. The static bounds and continuous nonzero $h$ give
\[
 \exists m,M,M_\rho>0\ \forall t\in J:\quad
 m\le\lambda(t)\le M,\qquad \|\rho(t)\|_1\le M_\rho.
\]
The preceding derivative bound makes $J\rho$ Lipschitz on $J$.
Kernel symmetry gives, for $f,g\in L^1$,
\begin{gather*}
 \mathcal G(s;f)-\mathcal G(s;g)=(Jf-Jg)(T_s(f+g)),\qquad
 \|T_sf\|_X\le\max\{1,3/s\}\|f\|_1,\\
 \exists C_J<\infty\ \forall t,z,v\in J:\quad
 |\mathcal G(\lambda(v);\rho(t))-\mathcal G(\lambda(v);\rho(z))|
 \le C_J|t-z|.
\end{gather*}
If $s=\lambda(t)\le\lambda(z)=q$, radial monotonicity gives the ordered-root
comparison
\begin{align*}
 q\vartheta s^3&\le\mathcal G(s;\rho(z)),\qquad
 \mathcal G(s;\rho(t))=\vartheta s^4,\\
 q-s&\le\frac{|\mathcal G(s;\rho(z))-\mathcal G(s;\rho(t))|}{\vartheta s^3}
 \le\frac{C_J}{\vartheta m^3}|t-z|.
\end{align*}
Exchange $t,z$ for the opposite ordering. Therefore
\[
 |\lambda(t)-\lambda(z)|\le\frac{C_J}{\vartheta m^3}|t-z|,\qquad t,z\in J,
 \qquad\lambda\in W^{1,\infty}_{\rm loc}(I).
\]
The identification of the ordinary a.e. derivative with the weak derivative
uses \cite{Hunter}.

\emph{5. Implicit identity and the signed time error.} For $s,q\ge m>0$,
direct Gaussian differentiation gives
\[
 \|T_sf-T_qf\|_X\le\max\{4/m,15/m^2\}|s-q|\|f\|_1.
\]
Thus the test $T_{\lambda(z)}(\rho(z)+\rho(t))$ is continuous in $X$.
Factor the density increment by the quadratic identity in Step~4 and
differentiate the fixed-density scale increment. On the common set where
$J\rho$ and $\lambda$ are differentiable, \eqref{selector-root-definition} gives
\begin{align*}
 4\vartheta\lambda^3\lambda'
 &=\partial_s\mathcal G(\lambda;\rho)\lambda'
       +2\mathcal D_t(T_\lambda\rho)\\
 &=\partial_s\mathcal G(\lambda;\rho)\lambda'
       +2(\mathcal J_\lambda+S),\qquad
 S:=\int r(\Psi_\lambda*\rho),\\
 A\lambda'&=2(\mathcal J_\lambda+S),\qquad
 A:=4\vartheta\lambda^3-\partial_s\mathcal G(\lambda;\rho)
       \ge\vartheta\lambda^3>0.
\end{align*}
The lower bound for $A$ uses $s\partial_s\mathcal G\le3\mathcal G$ at
the root. Hence
\begin{align*}
 \lambda'\mathcal J_\lambda
 &=\frac{2\mathcal J_\lambda(\mathcal J_\lambda+S)}A
 \ge-\frac{S^2}{2A}
 \ge-\frac{S^2}{2\vartheta\lambda^3},\\
 |S|^2&\le\|r\|_1^2\|\Psi_\lambda*\rho\|_\infty^2
 \le\|r\|_1^2\|\Psi_\lambda\|_2^2\|\rho\|_2^2
 =\pi^{3/2}\lambda^3a\|r\|_1^2,\\
 \lambda'\mathcal J_\lambda&\ge
 -\frac{\pi^{3/2}}{2\vartheta}\,a\|r\|_1^2.
\end{align*}
Here $x(x+y)\ge-y^2/4$ gives the first inequality. The common constants
\[
 C_\vartheta:=\max\{\Gamma(0)/\vartheta,\pi^{3/2}/(2\vartheta)\},\qquad
 c_\vartheta:=2\vartheta
\]
prove \eqref{selectorbounds}, \eqref{selectorQ} and \eqref{radiuserror}.
\end{proof}
The selector therefore provides the same density-determined scale for both the coercive term and the moving-kernel correction. We now apply it to the projected critical orbit.

\emph{Application to the same projected orbit.} The low-pass and inverse-gradient
estimates \cite{Tao} give
\begin{gather*}
 h=P_{>\nu}u\in C_tH^1_x,\qquad \|\nabla h\|_2\le2\sqrt G,\\
 \widetilde h(t,y):=N(t)^{-1/2}h(t,x(t)+y/N(t)),\qquad
 \operatorname{dist}_{\dot H^1}(\widetilde h(t),\K)\le\eta_\nu.
\end{gather*}
Proposition~\ref{spatial}, the compact fixed-ball lower bound and
Cauchy--Schwarz give constants $R_0,a_*>0$ such that, for all $t\ge0$ and
all sufficiently small $\nu$,
\[
 \|\widetilde h(t)\|_4^4=N(t)a(t)\ge a_*.
\]
The choices and their consequence are
\begin{gather*}
 R_0,a_*\ \longrightarrow\ L>0\text{ with }La_*\ge R_0
 \ \longrightarrow\ \vartheta=\vartheta(2\sqrt G,L),\\
 \lambda(t):=\Lambda_\vartheta(|h(t)|^2),\qquad
 N(t)\lambda(t)\ge LN(t)a(t)\ge R_0.
\end{gather*}
Affine change of variables, including the current-current and focusing
terms, gives the exact scaling identities
\[
 \mathfrak B_{k_{N\lambda}}[\widetilde h]
 =N\mathfrak B_{k_\lambda}[h],\qquad
 \Q_{N\lambda}(|\widetilde h|^2)=N\Q_\lambda(|h|^2).
\]
The Hessian, bilaplacian and pair-volume degrees are $-1,-3,6$.
Apply \eqref{spatialbound} at $N\lambda$, drop its positive remainder,
and use the same selected root:
\begin{equation}\label{quarticcoercivity}
 \mathfrak B_{k_\lambda}[h]
 \ge\frac{77}{81}\Q_\lambda(|h|^2)
 \ge\frac{154}{81}\vartheta\lambda
 \ge c_u a(t),\qquad c_u:=\frac{154}{81}\vartheta L>0.
\end{equation}
The constants are fixed before the final small cutoff $\nu$; no upper
bound for $N(t)a(t)$ is required.

For fixed $\nu>0$, the inverse-gradient and gradient-multiplier bounds imply
\[
 \|h(t)-h(s)\|_{H^1}\le C_\nu\|u(t)-u(s)\|_{\dot H^1}
 \qquad h\in C_tH^1_x
 \qquad p:=\im(\bar h\nabla h)\in C_tL^1_x.
\]
Put $l:=P_{\le\nu}u$, $w:=P_{\le\nu}F(u)$ and
$E:=F(u)-F(h)-w$. Bernstein
\cite{Tao} gives compatible realizations
\begin{gather*}
 P_{\le\nu}:L^1\to L^2,\quad\|P_{\le\nu}\|\le C\nu^{3/2},\qquad
 P_{\le\nu}:L^{6/5}\to L^2,\quad\|P_{\le\nu}\|\le C\nu,\\
 w\in C_tL^2_x,\qquad
 r:=-2\im(E\bar h)=-2\im\bigl((F(u)-w)\bar h\bigr)\in C_tL^1_x.
\end{gather*}
The last line uses $L^{6/5}L^6\subset L^1$ and $L^2L^2\subset L^1$.

For clarity, consider
\[
 v\in C_tH^1_x,
 \qquad F\in C_tL^{6/5}_x,
 \qquad w\in C_tL^2_x,
 \qquad F-w\in L^2_{t,\rm loc}H^{-1}_x,
\]
with $v_t=i\Delta v+i(F-w)$ in mild form. The Hilbert-triple product rule
\cite{Hunter}, applied to
$v$ and $\psi v$ for $\psi\in C_c^\infty(\R^3;\R)$ and followed by spatial
integration by parts, yields the distributional identity
\[
 \partial_t|v|^2=-2\divg\im(\bar v\nabla v)-2\im((F-w)\bar v).
\]
Schwartz testing of Duhamel gives the mild-to-weak implication
\cite{Tao}.
The smooth identity \cite{Tao} has thereby been extended
to the stated weak domain, with Laplacian coefficient $1$.
Commuting $P_{>\nu}$ with the free flow proves the projected mild formula;
the preceding implication gives $\rho_t=-2\divg p+r$ for the actual $h$.

For sufficiently small $\nu$, $N(t)a(t)\ge a_*>0$ on the forward ray.
Continuity at $0$ supplies an open interval $I\supset[0,\infty)$ on which
the original projected solution is nonzero. Apply the selector on this
interval. On $t\ge0$, the gauge identity and Lemma~\ref{L4lemma} give
\begin{align*}
 \im(F(u)\bar u)&=0
 \qquad r=2\im\bigl(F(u)\bar l+w\bar h\bigr),\\
 \|l\|_\infty&\le2\nu A_*,\qquad
 \|w\|_2\le C\nu^{3/2}\|F(u)\|_1,\\
 \|r\|_1&\le2\|F(u)\|_1\|l\|_\infty+2\|w\|_2\mathfrak m_\nu
 \le D_u(\nu+\nu^{1/2}\delta_\nu).
\end{align*}
Differentiating the explicit kernel, including its continuous value at zero,
\[
 \partial_R\nabla k_R=\nabla\Psi_R
 \qquad \mathcal J_\lambda
 =2\int p\cdot[(\partial_\lambda\nabla k_\lambda)*\rho].
\]
Thus the selector's current is the literal moving-kernel current. Finally,
\begin{gather*}
 \varepsilon_{\rm sc}(\nu):=C_u(\nu^2+\nu\delta_\nu^2)
 \longrightarrow0,\qquad
 \lambda\in W^{1,\infty}_{\rm loc}(I),
\\
 \lambda'(t)\mathcal J_\lambda(t)
 \ge-\varepsilon_{\rm sc}(\nu)a(t)
 \quad\text{for a.e. }t\ge0.
\end{gather*}

\section{Source Errors and Elimination of the Critical Element}
\subsection{Kernel bounds and separated frequencies}
For an even kernel $k\in C^3(\R^3\setminus\{0\})$, define
\[
 \|k\|_{\mathcal C_*^3}:=
 \sup_{x\ne0}\sum_{m=1}^3|x|^{m-1}\|D^mk(x)\|_{\rm op}.
\]
For the Gaussian family,
\begin{equation}\label{sourcekernel}
 \sup_{R>0}\|k_R\|_{\mathcal C_*^3}<\infty.
\end{equation}
Indeed, set $M_0:=1+\int_\R e^{-s^2/2}\,ds<\infty$. For $R>0$ and
$r=|x|>0$, direct differentiation of \eqref{gaussiankernel} gives
\begin{align*}
 r\Xi_R(r)&\le M_0,\qquad rL_R(r)\le3M_0,\qquad
 \|Dk_R(x)\|\le M_0,\\
 0\le H_R(x)[v,v]&\le\Xi_R(r)|v|^2
 \qquad\|D^2k_R(x)\|\le2\Xi_R(r),\\
 |D(H_R[v,v])(x)[d]|&\le12L_R(r)r^{-1}|d||v|^2
 \qquad\|D^3k_R(x)\|\le24L_R(r)/r,\\
 \sum_{m=1}^3r^{m-1}\|D^mk_R(x)\|
 &\le M_0+2M_0+72M_0=75M_0.
\end{align*}
The two implications use polarization on unit vectors. Evenness and
smoothness away from zero follow from the explicit Gaussian primitive.
Only this radius-uniform seminorm enters the spatial defect estimates.

\emph{Input defect and its exact normal form.} Keep $h=P_{>\nu}u$,
$l=P_{\le\nu}u$ and the full-density weights $q,q_m$. Define
\begin{gather*}
 \mathfrak c_\nu:=F(u)-F(h),\qquad
 D:=\tfrac23(|u|^6-|h|^6),\qquad \Pi:=\re(\mathfrak c_\nu\bar h),\\
 \mathfrak p_\nu:=4\int_0^1|h+\theta l|^4
       \re((h+\theta l)\bar l)\,d\theta-\re(F(u)\bar l).
\end{gather*}
First use smooth fields with the displayed products integrable. The
fundamental theorem, complex chain rule and gauge identity give
\begin{align*}
 D&=4\int_0^1|h+\theta l|^4\re((h+\theta l)\bar l)\,d\theta,\\
 \partial_jD&=4\re(\mathfrak c_\nu\partial_j\bar h)
       +4\re(F(u)\partial_j\bar l),\\
 \re(\mathfrak c_\nu\partial_j\bar h-h\partial_j\overline{\mathfrak c_\nu})
 &=2\re(\mathfrak c_\nu\partial_j\bar h)-\partial_j\Pi,\\
 D-2\Pi&=-2\mathfrak p_\nu,\qquad
 \im(\mathfrak c_\nu\bar h)=-\im(F(u)\bar l).
\end{align*}
Two spatial integrations by parts therefore yield
\begin{equation}\label{input-normal-form}
 \begin{aligned}
 \mathfrak S_k(\mathfrak c_\nu,h)
 &=\int(\divg q)(D-2\Pi)
       +4\re\int F(u)(q\cdot\nabla\bar l-iq_m\bar l)\\
 &=-2\int(\divg q)\mathfrak p_\nu
       +4\re\int F(u)(q\cdot\nabla\bar l-iq_m\bar l).
 \end{aligned}
\end{equation}
Every derivative now falls on the low input or the weight.

\emph{Actual weight and source bounds.} Put $m:=\|h\|_2$ and
$b:=\|\nabla h\|_2$. Then
\begin{align}
 \|q\|_\infty&\le M_0m^2,\qquad \|q_m\|_\infty\le3M_0mb,\label{full-weight-bounds}\\
 \divg q(x)&=\int|h(y)|^2\Delta k_R(x-y)\,dy,\notag\\
 |\divg q(x)|&\le3M_0\|h\|_2
       \bigl\||x-\cdot|^{-1}h\bigr\|_2
 \le6M_0mb,\notag
\\
 \|\divg q\|_\infty+\|q_m\|_\infty&\le9M_0mb.\notag
\end{align}
The last inequality uses \eqref{Hardy}. The inverse-distance bound is used
only off the null diagonal.
Uniform $L^5$ boundedness of the rescaled Schwartz convolution
\cite{Tao} gives,
for every real $\nu>0$,
\begin{align*}
 \|h\|_5&\le(1+C_5)\|u\|_5,\qquad
 f=\|u\|_5^5,\qquad \|F(h)\|_1\le(1+C_5)^5f,\\
 |\mathfrak p_\nu|&\le80|l|(|u|^5+|h|^5),\\
 \|\mathfrak p_\nu\|_1&\le80\bigl[1+(1+C_5)^5\bigr]\|l\|_\infty f.
\end{align*}
Insert these bounds into \eqref{input-normal-form}:
\begin{align*}
 |\mathfrak S_k(\mathfrak c_\nu,h)|
 &\le12M_0mb\bigl[80(1+(1+C_5)^5)+1\bigr]\|l\|_\infty f
       +12M_0m^2\|\nabla l\|_\infty f,\\
 \|l\|_\infty&\le C\nu\mathcal Y,\qquad
 \|\nabla l\|_\infty\le C\nu^2\mathcal Y,\qquad
 b\le2B,\qquad\nu m\le\delta_\nu,
\\
 |\mathfrak S_k(\mathfrak c_\nu,h)|
 &\le C_u(\delta_\nu+\delta_\nu^2)f\mathcal Y.
\end{align*}
All spatial integrals are finite: $\mathfrak p_\nu,F(u)\in L^1$ pair
with bounded weights, low input and weak gradient. The $L^\infty$ and
$L^6$ multiplier representatives and their derivatives agree in
$\mathcal S'$ and hence a.e.

\emph{Continuity before time integration.} The quintic difference bound is
\[
 \|F(v)-F(w)\|_1\le\|v-w\|_6
       (\|v\|_4+\|w\|_4)^2(\|v\|_6+\|w\|_6)^2.
\]
Together with the uniform $L^4$ bound and the original $L^6$ continuity,
\begin{gather*}
 F(u)\in C_tL^1_x,\qquad
 \mathfrak p_\nu=\tfrac23(|u|^6-|h|^6)-\re(F(u)\bar l)\in C_tL^1_x,\\
 l,\nabla l\in C_tL^\infty_x,\qquad
 |h|^2,\im(\bar h\nabla h)\in C_tL^1_x.
\end{gather*}
Bounded-kernel bilinear pairings are jointly continuous in these $L^1$
factors and in $R>0$ (dominated convergence on each compact radius range).
Therefore \eqref{input-normal-form}, evaluated at the actual continuous
$R=\lambda(t)>0$, is continuous and absolutely integrable on $[0,T]$.

\emph{Spatial extension to weak data.} For fixed $t$, $\nu,R>0$, let
\[
 u_M:=P_{\le M}u,\quad l_M:=P_{\le\nu}u_M,\quad h_M:=u_M-l_M,\qquad M\in\mathbb N_{>0}.
\]
Compact Fourier support and the classical/distributional derivative identities
\cite{Dyatlov} give smooth representatives. Bernstein and the product
rule imply
\begin{gather*}
 u_M,h_M\in L^6\cap L^\infty,\quad\nabla u_M,\nabla h_M\in L^2,\quad h_M\in L^2,\\
 F(u_M),F(h_M),\nabla F(u_M),\nabla F(h_M)\in L^2,\\
 D_M,\Pi_M,\nabla D_M,\nabla\Pi_M\in L^1,\qquad q_M,q_{m,M}\in C_b^1.
\end{gather*}
Here $D_M,\Pi_M$ are the above scalar defects with $(u,h,l)$ replaced by
$(u_M,h_M,l_M)$. All terms of both integrations by parts are integrable.
The cutoff symbols converge pointwise to one and are uniformly bounded;
Plancherel dominated convergence, the fixed highpass, and Sobolev give
\begin{gather*}
 u_M\to u\text{ in }\dot H^1\cap L^6,\qquad h_M\to h\text{ in }H^1,\\
 \sup_M\|u_M\|_4<\infty
 \quad\text{\cite{Tao}},\\
 F(u_M)\to F(u)\text{ in }L^1,\qquad
 l_M,\nabla l_M\to l,\nabla l\text{ in }L^\infty,\\
 \mathfrak p_{\nu,M},|h_M|^2,\im(\bar h_M\nabla h_M)
 \longrightarrow\mathfrak p_\nu,|h|^2,\im(\bar h\nabla h)
       \quad\text{in }L^1.
\end{gather*}
The same kernel-pairing bounds now give
\[
 \mathfrak S_{k_R}(F(u_M)-F(h_M),h_M)
 \longrightarrow \mathfrak S_{k_\lambda}(\mathfrak c_\nu,h).
\]
This defines the energy-class input source, independently of classical
representatives. Its participation in the time-dependent equation is justified below
by the energy-class limiting argument.

For the same $\nu$-independent times $I_j=[0,T_j]$, use
$f\mathcal Y\le f^2+\mathcal Y^2$ and \eqref{fcompare}:
\begin{equation}\label{inputdefect}
 \begin{aligned}
 |\mathfrak S_k(\mathfrak c_\nu,h)(t)|
 &\le C_u(\delta_\nu+\delta_\nu^2)f(t)\mathcal Y(t),\\
 \int_{I_j}|\mathfrak S_k(\mathfrak c_\nu,h)|\,dt
 &\le C_u(\delta_\nu+\delta_\nu^2)\mathfrak K_\nu(I_j).
 \end{aligned}
\end{equation}
Define the output defect with its sign:
\[
 \mathfrak o_\nu:=-P_{\le\nu}F(u),\qquad E=\mathfrak c_\nu+\mathfrak o_\nu.
\]
The input defect is now controlled; the remaining output defect is handled by the frequency gap.

\begin{proposition*}[Frequency-gap source estimate]
For $I_j=[0,T_j]$ and the same full-density weights $q,q_m$ as above,
\begin{equation}\label{gsource}
 \int_{I_j}|\mathfrak S_k(\mathfrak o_\nu,h)|dt\le C_u\sqrt{\delta_\nu}\,\mathfrak K_\nu(I_j)
\end{equation}
for all sufficiently small $\nu$. Its constant depends only on the fixed orbit and the seminorm threshold in \eqref{sourcekernel}.
\end{proposition*}
\begin{proof}
\emph{1. A linear slot with the weights held fixed.} The weights $q,q_m$
always use the complete $h$. For a weak $H^1$ slot $z$, define
\[
 \mathfrak S_k(\mathfrak o_\nu,z):=-2\int q\cdot
 \re(\mathfrak o_\nu\nabla\bar z-z\nabla\overline{\mathfrak o_\nu})
 -4\int q_m\im(\mathfrak o_\nu\bar z).
\]
This is real linear in $z$. The derivative-removed identity is
\begin{equation}\label{sourceIBP}
 \mathfrak S_k(\mathfrak o_\nu,z)
 =4\re\int\bar zq\cdot\nabla\mathfrak o_\nu
 +2\re\int\bar z\mathfrak o_\nu\divg q
 -4\im\int q_m\mathfrak o_\nu\bar z.
\end{equation}
For its weak-domain justification, Young--Bernstein
\cite{Tao} gives
\[
 \|\mathfrak o_\nu\|_2\le C\nu^{3/2}f,\qquad
 \|\nabla\mathfrak o_\nu\|_2\le C\nu^{5/2}f,\qquad
 \mathfrak o_\nu\in H^1.
\]
All $L^p$ and derivative realizations agree in $\mathcal S'$ and a.e.
Smooth both $H^1$ slots by $P_{\le M}$, keeping $q,q_m$ fixed. Compact
Fourier support \cite{Dyatlov} and
$L^2L^2\subset L^1$ imply that the source products and their first
derivatives belong to $L^1$, while $q,q_m\in C_b^1$.
Whole-space integration by parts therefore proves \eqref{sourceIBP}
for the smooth slots. Moreover,
\[
 P_{\le M}z\to z,\qquad P_{\le M}\mathfrak o_\nu\to\mathfrak o_\nu
 \quad\text{in }H^1.
\]
Continuity of bounded weighted $L^2$ pairings passes \eqref{sourceIBP}
to the original slots.

\emph{2. One near/far split.} Fix $\Lambda\ge8$ and set
\begin{gather*}
 H:=P_{>\Lambda\nu}u=P_{>\Lambda\nu}h,\qquad
 v:=h-H=(P_{\le\Lambda\nu}-P_{\le\nu})u,\qquad
 \mu:=\Lambda\nu/8,\\
 \|v\|_\infty\le C\Lambda\nu\mathcal Y,\qquad
 \|\nabla v\|_\infty\le C\Lambda^2\nu^2\mathcal Y,\qquad
 \|H\|_4\le Ca^{1/4},\quad\|H\|_6\le CB.
\end{gather*}
The identity $P_{>\Lambda\nu}P_{\le\nu}=0$ is a support identity.
Equation~\eqref{lowpoint} and the smooth multiplier bounds give the last
line for every positive $\nu$. Young--Bernstein
\cite{Tao}
also gives
\begin{equation}\label{gBernstein}
 \|\mathfrak o_\nu\|_1\le Cf,\quad
 \|\nabla\mathfrak o_\nu\|_1\le C\nu f,\quad
 \|\mathfrak o_\nu\|_{3/2}\le C\nu f,\quad
 \|\nabla\mathfrak o_\nu\|_{3/2}\le C\nu^2f.
\end{equation}
The near slot and its weak gradient are the differences of the corresponding
actual lowpass $L^\infty$ representatives. Thus every near product is in
$L^1$. Integrating the source derivative back onto $qv$ and using
\eqref{full-weight-bounds},
\begin{align*}
 |\mathfrak S_k(\mathfrak o_\nu,v)|
 &\le C\|\mathfrak o_\nu\|_1
 \bigl[\|q\|_\infty\|\nabla v\|_\infty
 +(\|\divg q\|_\infty+\|q_m\|_\infty)\|v\|_\infty\bigr]\\
 &\le C_B(\Lambda\delta_\nu+\Lambda^2\delta_\nu^2)f\mathcal Y.
\end{align*}

\emph{3. Actual derivative bounds for the far weights.} The density and
current have compatible representatives satisfying
\begin{align*}
 \rho&=|h|^2\in L^1\cap L^2,\qquad\|\rho\|_2=a^{1/2},\\
 p&=\im(\bar h\nabla h)\in L^1\cap L^{4/3},\qquad
 \|p\|_{4/3}\le3\|h\|_4\|\nabla h\|_2.
\end{align*}
At fixed $R>0$, the identities
\[
 \Xi_R(|x|)=\int_0^1\Phi_R(t|x|)\,dt,\qquad
 Dk_R(x)=\Xi_R(|x|)\langle x,\cdot\rangle
\]
give $k_R\in C^3$ globally. Its derivatives through order three are bounded
on the unit ball by continuity and outside it by \eqref{sourcekernel}.
Differentiation of the actual convolutions gives
\[
 D^2q_j=(D^2\partial_jk_R)*\rho.
\]
The derivative $\nabla q_m$ is taken from the original reversed convolution
kernel defining $q_m$.
The Hardy--Littlewood--Sobolev mappings
\[
 I_1:L^2\to L^6,
 \qquad
 I_2:L^{4/3}\to L^{12}
\]
from \cite{Tao}, together with
\eqref{sourcekernel}, imply
\[
 \|D^2q\|_6\le C\|\rho\|_2=Ca^{1/2},
 \qquad
 \|\nabla q_m\|_{12}\le C\|p\|_{4/3}\le CBa^{1/4}.
\]
These are nonendpoint mappings of the actual Fr\'echet derivatives.

Write $q^{\rm lo}:=P_{\le\mu}q$, $q^{\rm hi}:=P_{>\mu}q$, and likewise
for $q_m$. The bounded homogeneous Bernstein estimates
\cite{Tao} give
\[
 \|P_{>\mu}w\|_p\le C\mu^{-1}\|Dw\|_p
 \quad\text{for }w\in C_b^1,\ Dw\in L^p,
\]
and
\[
 \|P_{>\mu}w\|_p\le C\mu^{-2}\|D^2w\|_p
 \quad\text{for }w\in C_b^2,\ D^2w\in L^p.
\]
They follow from the mass-one lowpass kernel, its vanishing first moments,
Taylor remainders and Minkowski; $w\in L^p$ is not required. Applying them
with $p=6,12$ gives
\begin{equation}\label{far-weight-bounds}
 \begin{aligned}
 \|q^{\rm hi}\|_6&\le C(\Lambda\nu)^{-2}a^{1/2},\\
 \|\divg q^{\rm hi}\|_6&\le C(\Lambda\nu)^{-1}a^{1/2},\\
 \|q_m^{\rm hi}\|_{12}&\le C(\Lambda\nu)^{-1}Ba^{1/4}.
 \end{aligned}
\end{equation}
The highpass removes the zero-frequency polynomial ambiguity in
$\mathcal S'$; the estimates hold at every real $\mu>0$.

\emph{4. Exact low-weight cancellation and the three remaining terms.}
The supports satisfy
\begin{gather*}
 \supp\widehat{\mathfrak o_\nu}\subset\overline B_{2\nu},\qquad
 \supp\widehat H\subset\R^3\setminus B_{\Lambda\nu},\\
 \supp\widehat{q^{\rm lo}},\ \supp\widehat{q_m^{\rm lo}}
 \subset\overline B_{\Lambda\nu/4},\qquad 2+\Lambda/4<\Lambda.
\end{gather*}
Products of the bounded low weight with $\mathfrak o_\nu$ or
$\nabla\mathfrak o_\nu$ lie in $L^2$ and have disjoint Fourier support
from $H$. Therefore every low-weight pairing in \eqref{sourceIBP} is zero.
H\"older, \eqref{gBernstein} and \eqref{far-weight-bounds} give
\begin{align*}
 \left|\int\bar Hq^{\rm hi}\cdot\nabla\mathfrak o_\nu\right|
 &\le\|H\|_6\|q^{\rm hi}\|_6\|\nabla\mathfrak o_\nu\|_{3/2}
 \le CB\Lambda^{-2}fa^{1/2},\\
 \left|\int\bar H\mathfrak o_\nu\divg q^{\rm hi}\right|
 &\le\|H\|_6\|\divg q^{\rm hi}\|_6\|\mathfrak o_\nu\|_{3/2}
 \le CB\Lambda^{-1}fa^{1/2},\\
 \left|\int\bar Hq_m^{\rm hi}\mathfrak o_\nu\right|
 &\le\|H\|_4\|q_m^{\rm hi}\|_{12}\|\mathfrak o_\nu\|_{3/2}
 \le CB\Lambda^{-1}fa^{1/2}.
\end{align*}
The exponent identities are $1/6+1/6+2/3=1$ and $1/4+1/12+2/3=1$.

\emph{5. Time integration and one choice of $\Lambda$.} The source is
continuous along the actual positive $\lambda(t)$: its coefficients are
bounded Gaussian-kernel pairings of continuous $L^1$ density/current
and $L^1$ products of the source's $L^2$ factors. In particular,
\begin{gather*}
 |\mathfrak S_k(\mathfrak o_\nu,h)|\in L^1(I_j),\qquad
 f\mathcal Y\le f^2+\mathcal Y^2,\qquad f\sqrt a\le f^2+a.
\end{gather*}
Integrating these two inequalities and applying \eqref{fcompare} gives
\[
 \int_{I_j}f\mathcal Y+\int_{I_j}f\sqrt a
 \le C_u\mathfrak K_\nu(I_j).
\]
Real linearity in the slot, the near bound and the three far bounds imply
\[
 \int_{I_j}|\mathfrak S_k(\mathfrak o_\nu,h)|\,dt
 \le C_u(\Lambda^{-1}+\Lambda\delta_\nu+\Lambda^2\delta_\nu^2)
       \mathfrak K_\nu(I_j).
\]
Projected nonvanishing at $t=0$ gives $\delta_\nu>0$. By \eqref{smallfreq},
choose the cutoff range so that $0<\delta_\nu\le1/64$, and then set
\[
 \Lambda:=\delta_\nu^{-1/2}\ge8,\qquad
 \Lambda^{-1}=\Lambda\delta_\nu=\sqrt{\delta_\nu},\qquad
 \Lambda^2\delta_\nu^2=\delta_\nu\le\sqrt{\delta_\nu}.
\]
This proves \eqref{gsource} for the same $T_j$ and one time-independent
$\Lambda$. No limit assumption on $\Lambda\nu$ is used.
\end{proof}
The input and output source estimates are now expressed on the same terminal intervals and with the same full-density weights. We next justify the interaction identity directly at energy regularity.

\subsection{Passage to the energy class}
\emph{1. Local domains and the ordinary forced laws.} For a compact
$I\Subset I_u$ and fixed sufficiently small $\nu>0$, local theory gives
\begin{equation}\label{local-endpoint-domain}
 \nabla u\in L^2(I;L^6),\qquad
 h=P_{>\nu}u\in C(I;H^1)\cap L^2(I;W^{1,6}).
\end{equation}
Use \cite{KM}. For the native
$C_t\dot H^1$ mild solution, the identification follows from homogeneous
approximation \cite{TaoVisan} and uniqueness
\cite{TaoVisan}: the difference lies in
$CL^6\subset L^2(I;L^6)$ and has zero initial value. A finite compact-time
cover yields \eqref{local-endpoint-domain}, without individual finite mass.

Let $v,N\in C(J;H^1)$ on an open interval $J$ and suppose that
$v_t=i\Delta v+iN$ in mild form. Define
\[
 \rho:=|v|^2,
 \qquad
 j:=\im(\bar v\nabla v).
\]
The forced local laws are
\begin{equation}\label{forced-local-laws}
 \begin{aligned}
 \rho_t&=-2\divg j-2\im(N\bar v),\\
 \partial_tj_j&=-2\partial_k\re(\partial_jv\,\partial_k\bar v)
 +\tfrac12\partial_j\Delta\rho
 -\re(N\partial_j\bar v-v\partial_j\bar N).
 \end{aligned}
\end{equation}
Repeated spatial indices are summed. The normalization in
\cite{CKSTT}
is $iv_t+\Delta v=N_{\rm ref}$ with momentum density $2j$; here
$N_{\rm ref}=-N$. Smooth approximation in the linear Duhamel formula
gives convergence of every product in $CL^1$, and hence the displayed
$CH^1$ local laws.

For momentum testing use the normed graph
\[
 X_{\rm mom}:=\{\phi\in C^2:\phi,\Delta\phi\in C_b^1\},\qquad
 \|\phi\|_{X_{\rm mom}}:=\|\phi\|_{C_b^1}+\|\Delta\phi\|_{C_b^1}.
\]
Mollification and then spatial cutoff extend \eqref{forced-local-laws} to
this graph: remove the cutoff first at fixed mollification scale, then
pass to the limit against the $L^1$ products. The controlled quantities
are $\phi,\nabla\phi,\nabla\Delta\phi$.

\emph{2. Fixed-radius bilinear identity.} For fixed $R>0$,
\begin{gather*}
 k_R\in C^3,\quad\Delta k_R\in C_b^2,\quad
 \Delta\partial_jk_R=\partial_j\Delta k_R,\\
 \partial_jk_R*\rho\in X_{\rm mom},\qquad
 y\longmapsto\int\partial_jk_R(x-y)j_j(x)\,dx\in C_b^1.
\end{gather*}
Define
\begin{align*}
 K_R(v)&:=\iint\partial_{jk}k_R(x-y)\rho(y)
              \re(\partial_jv(x)\partial_k\bar v(x))\,dxdy,\\
 C_R(v)&:=\iint\partial_{jk}k_R(x-y)j_j(y)j_k(x)\,dxdy.
\end{align*}
Apply the momentum law to the first convolution test, the mass law to the
transposed test, and then the weak bilinear product rule and Fubini:
\begin{equation}\label{fixed-forced-balance}
 \frac d{dt}M_{k_R}(v)
 =\Q_R(\rho)+4K_R(v)-4C_R(v)+\mathfrak S_{k_R}(N,v).
\end{equation}
The Laplacian flux produces $\Q_R$; both forcing terms produce the complete
$\mathfrak S$, including the mass-current term. Bounded Gaussian derivative
kernels and $L^1$ factors make every pairing finite. Continuous local dual
derivatives and the Banach fundamental theorem
\cite{Hunter} give the derivative at each interior
time and its integrated endpoint identity.

On positive-time compact intervals, apply this with the actual fields
\[
 v=h_M:=P_{>\nu}P_{\le M}u,\qquad
 N=N_M:=P_{>\nu}P_{\le M}F(u).
\]
The previously constructed $C_tL^1_x$ forcing and differentiated lowpass
bounds give $N_M\in C_tH^1_x$.

\emph{3. A moving positive Lipschitz radius.} Temporarily let $R=R(t)>0$
be Lipschitz on a compact time interval, with $R(t)\ge m>0$. Then
\[
 \partial_R\nabla k_R=\nabla\Psi_R,\qquad
 \|\nabla\Psi_R\|_\infty\le3/m.
\]
The mixed radius--field increment is bounded by
\begin{align*}
 C|R(t+s)-R(t)|
 \bigl(&\|\rho(t+s)-\rho(t)\|_1\|j(t+s)\|_1\\
       &+\|\rho(t)\|_1\|j(t+s)-j(t)\|_1\bigr)=o(|s|).
\end{align*}
For held density and current, dominated differentiation gives the literal
$\mathcal J_R$. Combining this with \eqref{fixed-forced-balance},
\begin{equation}\label{moving-forced-balance}
 \frac d{dt}M_{k_{R(t)}}(v(t))
 =\Q_{R(t)}(\rho)+4K_{R(t)}(v)-4C_{R(t)}(v)
  +\mathfrak S_{k_{R(t)}}(N,v)+R'(t)\mathcal J_{R(t)}
\end{equation}
for a.e. $t$. The fixed-radius rate is jointly continuous in $R>0$ and
the two $H^1$ fields. Its compact uniform bound, the held-radius mean
value theorem and the radius Lipschitz estimate imply
that $t\mapsto M_{k_{R(t)}}(v(t))$ is Lipschitz, hence absolutely
continuous. The fundamental theorem of calculus therefore gives the
integrated form of \eqref{moving-forced-balance}.
Continuity of the action and integrability of the rate extend positive-time
identities to an endpoint at $0$. No negative-time extension of the
regularized forcing equation is used. From now on, set
\[
 R(t)=\lambda(t):=\Lambda_\vartheta(|P_{>\nu}u(t)|^2).
\]
This scale is independent of the regularization parameter $M$.
Its positivity and Lipschitz regularity were proved from the original
projected mass law and the same compact-orbit nonvanishing.

\emph{4. Focusing term and the actual smoothed defect.} For smooth bounded
$h\in L^6$ with $\nabla h\in L^2$, the complex polynomial gives
\begin{gather*}
 \im(F(h)\bar h)=0,\qquad
 \re(F(h)\partial_j\bar h-h\partial_j\overline{F(h)})
       =-\partial_j(\tfrac23|h|^6),\\
 |h|^6,\nabla(|h|^6)\in L^1,\qquad q\in C_b^1,\\
 \mathfrak S_{k_R}(F(h),h)=-\tfrac43\int(\divg q)|h|^6.
\end{gather*}
The last identity is whole-space integration by parts. Entrywise expansion
and the Hessian trace give
\begin{equation}\label{principal-expanded}
 \mathfrak B_{k_R}[h]=\Q_R(\rho)+4K_R(h)-4C_R(h)
                     -\tfrac43\int(\divg q)|h|^6.
\end{equation}
All terms are bounded-kernel $L^1$ pairings; in particular
$|h|^6=\re(F(h)\bar h)$ is a continuous $L^1$ product.

From now on fix $b>0$ and $I=[0,b]$; all following forcing bounds
and regularization limits are on this forward interval. Set
\begin{gather*}
 u_M:=P_{\le M}u,\quad h_M:=P_{>\nu}u_M,\quad
 c_M:=P_{\le M}F(u)-F(u_M),\quad r_M:=P_{>\nu}c_M,\\
 N_M-F(h_M)
 =[F(u_M)-F(h_M)]-P_{\le\nu}F(u_M)+r_M.
\end{gather*}
This is an exact multiplier identity; $u_M$ always has its actual smoothed
forcing $P_{\le M}F(u)$. Young's bounds
\cite{Tao} and finite-$p$ approximation
\cite{TaoApprox}, only for $p=6/5,6$, imply
\begin{align*}
 P_{\le M}F(u)&\longrightarrow F(u)\quad\text{in }L^2(I;L^{6/5}),\\
 u_M(t)&\longrightarrow u(t)\quad\text{in }L^6,\\
 \sup_{M,t\in I}\|u_M(t)\|_6&<\infty,\\
 F(u_M)&\longrightarrow F(u)\quad\text{in }L^2(I;L^{6/5}),\\
 c_M,\ r_M&\longrightarrow0\quad\text{in }L^2(I;L^{6/5}),\\
 h_M&\longrightarrow h\quad\text{in }L^2(I;W^{1,6}),\\
 h_M(t)&\longrightarrow h(t)\quad\text{in }H^1,\\
 \sup_{M,t\in I}\bigl(\|h_M(t)\|_2+\|\nabla h_M(t)\|_2\bigr)&<\infty.
\end{align*}
The second line uses dominated convergence; the fourth uses
\eqref{local-endpoint-domain} on the actual weak gradient.

\emph{5. The complete commutator vanishes.} Form $q_M,q_{m,M}$ from the
complete $h_M$ at the original $\lambda(t)$. For a source slot
$Z\in L^{6/5}$ and $v,\nabla v\in L^6$, use the finite expression
\begin{equation}\label{weak-source-form}
 \mathfrak S_k^{\flat}(Z,v):=
 -4\int q\cdot\re(Z\nabla\bar v)
 -2\int(\divg q)\re(Z\bar v)
 -4\int q_m\im(Z\bar v).
\end{equation}
It equals the classical source by integration by parts and is real linear
in $Z$. Since lowpass contracts the projected $L^2$ mass and gradient,
\eqref{full-weight-bounds} gives constants uniform in $M$, and
\begin{align*}
 \int_I|\mathfrak S_{k_\lambda}^{\flat}(r_M,h_M)|\,dt
 &\le C_{I,\nu}\|r_M\|_{L^2(I;L^{6/5})}
       \bigl(\|\nabla h_M\|_{L^2(I;L^6)}+\|h_M\|_{L^2(I;L^6)}\bigr)\\
 &\longrightarrow0.
\end{align*}
Joint continuity of the Gaussian pairings proves measurability for the
strongly measurable gradient representatives. The $H^1$ and $L^{6/5}$
realizations of $N_M$ agree in $\mathcal S'$ and a.e. Thus the four pieces
of $N_M$ give respectively
\begin{align*}
 F(h_M)&\longmapsto\eqref{principal-expanded},\\
 F(u_M)-F(h_M)&\longmapsto\eqref{input-normal-form},\\
 -P_{\le\nu}F(u_M)&\longmapsto\eqref{sourceIBP},\\
 r_M&\longmapsto0.
\end{align*}
No separate $H^1$ realization of a nonlinear defect is required.
The $L^4$ lowpass bound \cite{Tao} gives
$F(u_M)\in C_tL^1_x$ with one uniform bound. All identities and limits
use the same weak-gradient representative on $[0,b]$ and the same original
selector; an extension used for measurability does not alter $\lambda'$.

\emph{6. Every term has an $L^1_t$ limit.} Keep $I=[0,b]$ and choose
$m,A,B>0$ with
\[
 \lambda\ge m,\qquad \|h_M\|_2\le A,\qquad\|\nabla h_M\|_2\le B
 \quad\text{on }I\text{ for all }M.
\]
If $\|v\|_6\le S\|\nabla v\|_2$, the bounded Gaussian kernels give
\begin{gather*}
 |\mathfrak B_{k_\lambda}[h_M]|
 \le\frac{7A^4}{m^3}+\frac{144A^2B^2}{m}
       +\frac{4A^2(SB)^6}{m},\qquad
 |\mathcal J_\lambda(h_M)|\le\frac{18A^3B}{m}.
\end{gather*}
At each time, $H^1\cap L^6$ convergence gives
\begin{align*}
 |h_M|^2&\longrightarrow |h|^2,\\
 \im(\bar h_M\nabla h_M)&\longrightarrow\im(\bar h\nabla h),\\
 \re(\nabla\bar h_M\otimes\nabla h_M)&\longrightarrow
 \re(\nabla\bar h\otimes\nabla h),\\
 |h_M|^6&\longrightarrow|h|^6
\end{align*}
in $L^1_x$.
For the principal and input terms, dominated convergence gives
\begin{align*}
 \mathfrak B_{k_\lambda}[h_M]&\to\mathfrak B_{k_\lambda}[h]
       \quad\text{in }L^1(I),\\
 \mathfrak S_{k_\lambda}(F(u_M)-F(h_M),h_M)
 &\to\mathfrak S_{k_\lambda}(\mathfrak c_\nu,h)
       \quad\text{in }L^1(I).
\end{align*}
For the output term,
\[
 F(u_M)\to F(u)\quad\text{in }L^1_x,
 \qquad \sup_M\|F(u_M)\|_1<\infty,
\]
so
\[
 -P_{\le\nu}F(u_M)\to\mathfrak o_\nu\quad\text{in }H^1,
\]
and the associated source terms converge in $L^1(I)$. For the moving term,
\begin{align*}
 |\lambda'\mathcal J_\lambda(h_M)|
 &\le(18A^3B/m)|\lambda'|\in L^1(I),\\
 \lambda'\mathcal J_\lambda(h_M)
 &\to\lambda'\mathcal J_\lambda(h)\quad\text{in }L^1(I).
\end{align*}
Each limit uses dominated convergence and the bounded-kernel pairing
representation. For the input limit, \eqref{input-normal-form},
\eqref{full-weight-bounds} and $\sup_M(\|u_M\|_4+\|u_M\|_6)<\infty$
give the spatial extension already proved.
Together with Step~5, the entire rate converges in $L^1(I)$.
Both endpoint actions converge by their bounded-kernel $L^1$ pairings.

\emph{7. Identify the original weak source.} In $L^{6/5}$ and $\mathcal S'$,
\[
 E=F(u)-F(h)-P_{\le\nu}F(u)=\mathfrak c_\nu+\mathfrak o_\nu.
\]
For a.e. time, \eqref{local-endpoint-domain} gives $\nabla h\in L^6$.
Smoothing that same weak gradient, the $L^{6/5}$--$L^6$ pairings in
\eqref{weak-source-form} and the classical input identity give
\[
 \mathfrak S_{k_\lambda}^{\flat}(E,h)
 =\mathfrak S_{k_\lambda}(\mathfrak c_\nu,h)
  +\mathfrak S_{k_\lambda}(\mathfrak o_\nu,h)
 =:\mathfrak S_{k_\lambda}(E,h)\quad\text{a.e.}
\]
The output identification uses its actual $H^1$ realization; the sum is
the integrable rate limit. Passing $M\to\infty$ first on $[0,b]$ and then
subtracting endpoint identities proves, for every $s,t\ge0$ in either order,
\begin{equation}\label{integratedidentity}
 M_{k_{\lambda(t)}}(h(t))-M_{k_{\lambda(s)}}(h(s))
 =\int_s^t\bigl[
 \mathfrak B_{k_{\lambda(\tau)}}[h]
 +\mathfrak S_{k_{\lambda(\tau)}}(E,h)
 +\lambda'\mathcal J_\lambda\bigr]\,d\tau.
\end{equation}

\subsection{Quartic integrability and the zero-mass contradiction}
\emph{1. Absorb on the single terminal sequence.} For each fixed $\nu>0$,
\begin{gather*}
 |M_{k_\lambda}(h)|\le C\|h\|_2^3\|\nabla h\|_2\quad(\lambda>0),\\
 C_\nu:=2\sup_{t\ge0}|M_{k_{\lambda(t)}}(h(t))|<\infty.
\end{gather*}
Keep the spatial constants, $L$, $\vartheta$, and the same $T_j,C_{\rm cmp}$
before decreasing $\nu$. Define
\begin{align*}
 \varepsilon_\nu:=
 &C_{\rm in}C_{\rm cmp}^2(\delta_\nu+\delta_\nu^2)
 +C_{\rm out}\sqrt{\delta_\nu}
 +C_{\rm mov}(\nu^2+\nu\delta_\nu^2)
 \longrightarrow0.
\end{align*}
The constants are independent of $j,\nu$. Equations
\eqref{quarticcoercivity}, \eqref{inputdefect}, \eqref{gsource}, the signed
moving bound and \eqref{integratedidentity} give
\[
 c_u\mathfrak K_\nu(I_j)
 \le C_\nu+\varepsilon_\nu\mathfrak K_\nu(I_j).
\]
Choose $\nu$ below all preceding thresholds and so that
$\varepsilon_\nu<c_u/2$. Then, for every $j$,
\[
 \mathfrak K_\nu(I_j)\le2C_\nu/c_u.
\]
Since $T_j\uparrow\infty$ and the integrand is nonnegative,
\[
 \mathfrak K_\nu([0,\infty))<\infty.
\]
Finally,
\[
 \int_{I_j}(f^2+\mathcal Y^2)
 \le C_{\rm cmp}^2\mathfrak K_\nu(I_j),
\]
so $(f,\mathcal Y)\in L^2(0,\infty;\R^2)$.
The last implication is a second monotone-convergence passage with the
same comparison witness.

\emph{2. Construct a low-frequency $L^2$ value.} Fix $K>0$ and set
$Q_K:=P_{\le K}$. The already used $L^1\to L^1$ and $L^1\to L^2$ bounds give
\begin{gather*}
 \|Q_Kg\|_{6/5}\le\|Q_Kg\|_1^{2/3}\|Q_Kg\|_2^{1/3}
 \le A_K\|g\|_1,\qquad A_K<\infty,\\
 \mathbf1_{[t,\infty)}Q_KF(u)\in L^2_sL^{6/5}_x,\qquad
 \|Q_KF(u)\|_{L^2([t,\infty);L^{6/5})}\le A_K\|f\|_{L^2(t,\infty)}.
\end{gather*}
The multiplier realizations agree in $\mathcal S'$. For this fixed $K$,
dual endpoint Strichartz
\cite{KT,KM} constructs
\begin{equation}\label{low-L2-tail}
 \begin{gathered}
 m_K(t):=-i\int_t^\infty e^{i(t-s)\Delta}Q_KF(u(s))\,ds\in L^2,\\
 \|m_K(t)\|_2\le CA_K\|f\|_{L^2(t,\infty)}.
 \end{gathered}
\end{equation}
The integral is defined through its dual action on Schwartz tests.
Apply the original \eqref{no-waste-formula} to the same cutoff tests.
Commutation gives
\begin{gather*}
 Q_Ke^{it\Delta}=e^{it\Delta}Q_K,\qquad Q_K\partial_j=\partial_jQ_K,\\
 \partial_jm_K(t)=Q_K\partial_ju(t)\quad\text{in }\mathcal S'.
\end{gather*}
Sobolev embedding \cite{Hunter} gives
\[
 m_K(t)\in L^2,\quad\nabla m_K(t)\in L^2,\qquad m_K(t)\in H^1\subset L^6.
\]
Finally,
\begin{gather*}
 m_K(t)-Q_Ku(t)\in L^6,\qquad\nabla(m_K(t)-Q_Ku(t))=0,\\
 m_K(t)=Q_Ku(t).
\end{gather*}
The last implication is uniqueness of the $L^6$ representative with its
homogeneous weak derivatives; no additive constant remains.

\emph{3. Mass conservation only after construction.} The complementary
highpass gives
\[
 \|(I-Q_K)u(t)\|_2\le CK^{-1}\|\nabla u(t)\|_2\le CB/K.
\]
At $K=1$, this and \eqref{low-L2-tail} imply
\[
 u(t)=Q_1u(t)+(I-Q_1)u(t)\in L^2\qquad(t\ge0).
\]
Local $H^1$ theory and homogeneous uniqueness
\cite{TaoVisan}
identify the same original solution with the $H^1$ solution. The mass law
and approximation \cite{Tao} then give
\[
 \|u(0)\|_2=\|u(t)\|_2
 \le CA_K\|f\|_{L^2(t,\infty)}+CB/K.
\]
For fixed $K$, letting $t\to\infty$ yields
\[
 \|u(0)\|_2\le CB/K.
\]
Letting $K\to\infty$ gives $\|u(0)\|_2=0$, hence $u(0)=0$ in
$\dot H^1$ and therefore $E_c=0$.
This contradicts \eqref{standing-orbit-data}. The time limit precedes the
frequency limit; no uniform estimate on $A_K$ as $K\to\infty$ is needed.

\emph{4. Return to the original theorem.} Let $u$ be an arbitrary
maximal strong solution with
\[
 E(u_0)<E(W),\qquad\|\nabla u_0\|_2<\|\nabla W\|_2.
\]
If $I_u\ne\R$ or $\|u\|_{L^{10}(I_u\times\R^3)}=\infty$, the critical-element
reduction produces an orbit satisfying \eqref{standing-orbit-data}, which
has just been ruled out. Therefore
\[
 I_u=\R,
 \qquad
 \|u\|_{L^{10}(\R\times\R^3)}<\infty.
\]
Apply the finite-size criterion separately on the two half-lines
\cite{KM}:
\begin{gather*}
 \exists u_+,u_-\in\dot H^1(\R^3),\\
 \lim_{t\to+\infty}\|u(t)-e^{it\Delta}u_+\|_{\dot H^1}=0,\\
 \lim_{t\to-\infty}\|u(t)-e^{it\Delta}u_-\|_{\dot H^1}=0.
\end{gather*}
These are the two actual homogeneous gradient-norm limits in
Theorem~\ref{mainthm}. This completes the proof of Theorem~\ref{mainthm}.

\section*{Acknowledgments}
ChatGPT (OpenAI) was used to formulate research ideas and proof strategies,
develop and check mathematical derivations, and assist with editing and
Lean formalization. The authors reviewed the resulting arguments and take
full responsibility for the mathematics. The full-manuscript Lean
formalization is maintained in a
\href{https://github.com/penghongzhong/3D-NLS-Lean-Verification}{private GitHub repository};
its proof sources have passed compilation and terminal audits in GitHub
Actions, relative to 51 explicitly documented standard literature interfaces
retained as theorem parameters.

\bibliographystyle{amsplain}
\bibliography{references}
\end{document}